\documentclass[10pt]{article}
\usepackage[utf8]{inputenc}
\usepackage[a4paper, left=2cm,
  right=1.5cm,
  top=2.5cm,
  bottom=2.5cm]{geometry}

\usepackage{amsmath} 
\usepackage{bbm}
\usepackage{amsfonts,amssymb,amsthm}
\usepackage{graphicx}
\usepackage[colorlinks,linkcolor=Blue,
            anchorcolor=red,
            citecolor=orange,]{hyperref}
\usepackage{mathrsfs}
\usepackage{enumerate}
\usepackage{indentfirst}
\usepackage{geometry}
\usepackage{url}
\usepackage[dvipsnames,svgnames,table]{xcolor}

\allowdisplaybreaks[4]
\numberwithin{equation}{section}

\newtheorem{definition}{Definition}[section]
\newtheorem{theorem}[definition]{\textbf{Theorem}}
\newtheorem{lemma}[definition]{\textbf{Lemma}}
\newtheorem{proposition}[definition]{\textbf{Proposition}}
\newtheorem{corollary}[definition]{\textbf{Corollary}}
\newtheorem{remark}[definition]{\textbf{Remark}}

\newcommand{\dd}{\mathrm{d}}
\newcommand{\ind}{\mathbbm{1}}

\newcommand{\bN}{\mathbb{N}}

\newcommand{\bR}{\mathbb{R}}

\newcommand{\bE}{\mathbb{E}} %
\newcommand{\bP}{\mathbb{P}} 

\newcommand{\cB}{\mathcal{B}} 
\newcommand{\cR}{\mathcal{R}} 
\newcommand{\cH}{\mathcal{H}} 

\newcommand{\cL}{\mathcal{L}} 
\newcommand{\cS}{\mathcal{S}} 
\newcommand{\cX}{\mathcal{X}} %
\newcommand{\cW}{\mathcal{W}} %
\newcommand{\cY}{\mathcal{Y}} 
\newcommand{\cZ}{\mathcal{Z}} 
\newcommand{\cM}{\mathcal{M}} 
\newcommand{\cJ}{\mathcal J}
\newcommand{\tcR}{\widetilde{\mathcal{R}}}

\newcommand{\rB}{\mathscr{B}}
\newcommand{\rF}{\mathscr{F}} 
\newcommand{\rG}{\mathscr{G}} 
\newcommand{\rE}{\mathscr{E}}

\newcommand{\rL}{\mathscr{L}}
\newcommand{\rS}{\mathscr{S}}

\newcommand{\eqdis}{\overset{d}{=}}

\newcommand{\fM}{\mathbf{M}} 

\newcommand{\tX}{\widetilde{X}}

\newcommand{\cT}{\mathcal{T}} 

\renewcommand{\emph}[1]{\textbf{#1}}

\title{An infinite-dimensional representation of 
Ray--Knight theorems for the Bass--Burdzy flow 
}
\author{Chengshi Wang\footnote{SMS, Fudan University, \texttt{cswang17@fudan.edu.cn}}}
\date{}

\begin{document}

\maketitle

\begin{abstract}
    In \cite{bass1999stochastic}, Bass and Burdzy introduced a stochastic flow of homeomorphisms of the real line $\cR_t:x\mapsto \cR_t(x)$, driven by a Brownian motion.
    Ray--Knight theorems for the process $\cR_t^{-1}(0)$ were later obtained by Hu and Warren \cite{hu-warren}. We extend their results by proving an integral representation driven by a common white noise for the local times of $\cR_t^{-1}(0)$, in analogy with the case of the standard Ray--Knight theorems \cite{aidekon2024infinite}. We generalize this representation to the setting of a Bass--Burdzy flow driven by a drifted Brownian motion perturbed at its maximum, and use this framework to answer a question of A\"id\'ekon, Hu and Shi \cite{aidekon2023stochastic} concerning the description of the conditional law of a perturbed reflecting Brownian motion given its occupation measure in terms of Bass--Burdzy flows.
\end{abstract}



\section{Introduction}\label{s:introduction}

Let $(X_t: t\ge 0)$ be a continuous function with $X_0=0$. Consider the equation
\begin{align}\label{eq:bbflow}
    \cR_t(x) = x + X_t + \lambda \int_0^t \ind_{\{\cR_s(x)> 0\}} \dd s,\qquad t\ge 0,
\end{align}
where $\lambda\in\{-1,1\}$.
Bass and Burdzy~\cite{bass1999stochastic} investigated the existence and uniqueness of strong solutions to~\eqref{eq:bbflow} under suitable assumptions on the driving function $X$. The case 
\begin{equation}\label{drifted case}
X_t=\beta_t +\delta t,
\end{equation}
where $\beta$ is a standard Brownian motion and $\delta\in\mathbb{R}$, has attracted the most attention. In this case \eqref{eq:bbflow} defines a stochastic flow of homeomorphisms $(\cR_t(x):t\ge 0,\, x\in\bR)$ known as the Bass--Burdzy flow\footnote{The Bass--Burdzy flow is usually written as $\mathrm{d} \cR_t=\mathrm{d} \beta_t + \lambda_1 \ind_{\{\cR_t(x)> 0\}} \mathrm{d}t+\lambda_2 \ind_{\{\cR_t(x)\le 0\}}\mathrm{d}t$. We can write it as $\mathrm{d} \cR_t=\mathrm{d} (\beta_t+\lambda_2t) + (\lambda_1-\lambda_2) \ind_{\{\cR_t(x)> 0\}}\dd t$ and, when $\lambda_1\ne\lambda_2$, use the Brownian scaling to reduce it to \eqref{eq:bbflow}.}: almost surely, each $(\cR_t(x):t\ge 0)$ solves \eqref{eq:bbflow} and $x\mapsto \cR_t(x)$ is continuous and strictly increasing for all $t\ge 0$. We can distinguish three different regimes as follows.
\begin{enumerate}[(i)]
    \item (Transience) if $\min(\delta,\delta+\lambda) > 0$ or $(\delta, \delta+\lambda) = (0,1)$, then for all $x\in\bR$, $\lim_{t\to\infty}\cR_t(x)= \infty$. If $\max(\delta,\delta+\lambda) < 0$ or $(\delta,\delta+\lambda)= (-1,0)$, then   for all $x\in\bR$, $\lim_{t\to\infty}\cR_t(x)= -\infty$.
    
    \item (Recurrence) If $\delta \in [0,1]$ and $\lambda=-1$, then $\cR(x)$ returns to $0$ at arbitrarily large times for all $x\in\bR$.
    
    \item (Bifurcation) If $\delta \in (-1,0)$ and $\lambda=1$, then there is a unique random $x_0\in\bR$ such that $\lim_{t\to\infty}\cR_t(x_1)= \infty$ for all $x_1>x_0$ while $\lim_{t\to\infty}\cR_t(x_2)= -\infty$ for all $x_2<x_0$. 
\end{enumerate}

\paragraph{Ray--Knight theorems}
For each $x\in\bR$, let $L_t^x$ be the semimartingale local time of $\cR_t(x)$ at time $t$ and position $0$, which, as shown in \cite{hu-warren}, admits a version that is bicontinuous in $(t,x)$. We emphasize that the local time is indexed by the initial position of the solution to \eqref{eq:bbflow}. A classical topic concerning local times is to derive Ray--Knight theorems, i.e. the law of $x\mapsto L_T^x$ for some typical stopping time $T$.
Bass and Burdzy \cite[Theorem 3.8]{bass1999stochastic} when $\delta>1$, $\lambda=-1$ and subsequently
Hu and Warren \cite{hu-warren} for general $(\delta,\lambda)$ established the following Ray--Knight theorem. 
Let $x^+ = \max\{x,0\}$ and $x^- =\max\{-x,0\}$, and 
\begin{align}\label{eq:tau_a}
\tau_a(x):= \inf\{t>0:L_t^x>a\}.
\end{align}

\begin{theorem}[Transient and recurrent regimes, {\cite[Theorems 1.1--1.2]{hu-warren}}]\label{thm:law HW}
Let $\delta\ge 0,\, \lambda\in\{1,-1\}$. 
Define $\kappa(\ell)=\lambda(e^{\lambda\ell}-1)$. Then for every $x\in\bR$ and $a\ge 0$, conditionally on $\{\tau_a(x)<\infty\}$, the process $\big(L_{\tau_a(x)}^{x+h}: h\ge 0\big)$ is a diffusion starting at $a$ with the infinitesimal generator
    \begin{align}\label{eq:generator L_tau HW}
        2\kappa(\ell)\frac{\dd^2}{\dd \ell^2} + 2 \bigg( \ind_{[0,x^-]}(h) - \delta\, \kappa(\ell) \bigg) \frac{\dd}{\dd \ell}.
    \end{align}
The process $\big(L_{\tau_a(x)}^{x+h}: h\ge x^-\big)$ is absorbed at 0.

Moreover, in the transient regime, $L_\infty^x$ is finite for all $x\in\bR$, and the random variable $L_\infty^0$ has the exponential distribution of parameter $\delta+\lambda$. Conditionally on $\{L_\infty^0 =b \}$, the two processes $(L_\infty^y: y\ge 0)$ and $(L_\infty^{-y}: y\ge 0)$ are independent, both start at $b$, and have infinitesimal generators respectively
    \begin{align}\label{eq:generator L_infty HW}
        2\kappa(\ell)\frac{\dd^2}{\dd \ell^2} - 2\delta\, \kappa(\ell) \frac{\dd}{\dd \ell},\quad 2\kappa(\ell)\frac{\dd^2}{\dd \ell^2} + 2 \bigg(1 - \delta\, \kappa(\ell) \bigg) \frac{\dd}{\dd \ell}.
    \end{align}
The process $(L_\infty^y: y\ge 0)$ is absorbed at 0.
\end{theorem}
A Ray--Knight theorem for the bifurcating regime is also given in \cite[Theorem 1.3]{hu-warren}, which, as noted there, can actually be deduced from the recurrent regime by a time-reversal argument. 

For every $t\ge 0$, we define
\begin{equation}\label{def:Y}
    Y_t:=\cR^{-1}_t(0)
\end{equation}
as the unique $x\in\bR$ such that $\cR_t(x)=0$. Then Theorem~\ref{thm:law HW} can be interpreted as a Ray--Knight theorem for the local times of $Y$. Indeed, it is shown in \cite[Proposition 1.2]{hu-warren} that  $Y$ admits a bicontinuous occupation density given by $\rho(t,x):=\kappa(L_t^x)$. 
Our first goal is to give a multi-dimensional analog of Theorem \ref{thm:law HW}.  To this end, we follow the approach of A\"id\'ekon, Hu and Shi \cite{aidekon2024infinite} and derive a stochastic integral representation of the local times $L^x_t$, driven by a \emph{white noise} constructed from the process $Y$; see also \cite{li2025stochastic} for a similar representation of the Ray–Knight theorem of the L\'evy forest.
We briefly review stochastic integration with respect to a white noise in Section~\ref{s:WN}, relying on Walsh \cite{Walsh1986AnIT}. The stochastic integral equations read as follows.

\begin{theorem}\label{thm:rk drifted}
Let $\delta\ge 0,\, \lambda\in\{1,-1\}$. 
Define the white noise $\cW$ via \eqref{eq:WN}. Then for every $x\in\bR$ and $a\ge 0$, almost surely, on the event $\{\tau_a(x)<\infty\}$,
    \begin{align}\label{eq:SDE L_tau drifted}
        L_{\tau_a(x)}^y = a - 2\int_x^y \cW([0,\kappa(L_{\tau_a(x)}^z)],\dd z) + 2 (x^- - y^- ) - 2\delta \int_x^y \kappa(L_{\tau_a(x)}^z)\dd z,\qquad y\ge x.
    \end{align}
Moreover, in the transient case, almost surely,
    \begin{align}\label{eq:SDE L_infty drifted}
        L_{\infty}^y = L_\infty^x - 2\int_x^y \cW([0,\kappa(L_{\infty}^z)],\dd z) + 2 (x^- - y^- ) - 2\delta \int_x^y \kappa(L_{\infty}^z)\dd z,\qquad  y\ge x.
    \end{align}
\end{theorem}

We define the time-change
\begin{align}\label{eq:zeta}
\zeta(t) := \int_0^t \frac{\dd s}{\big(1 + \rho(s,Y_s)\big)^2},\qquad t\ge 0.
\end{align}
Then $t\mapsto\zeta(t)$ is continuous and strictly increasing. Its limit as $t\to\infty$ is infinite in the transient and recurrent regimes from the occupation times formula and Theorem~\ref{thm:law HW}, while finite in the bifurcating regime by results in \cite[Section 7.2]{aidekon2023stochastic}. 
We focus on the case $\lambda=1$ and define 
\begin{equation}\label{def:Yhat}
    \widehat{Y}_t:=Y_{\zeta^{-1}(t)},\qquad 0\le t<\zeta(\infty).
\end{equation}
We also define 
\begin{equation}\label{def:Lhat}
    \widehat{L}_t^x:= 1-e^{-L_t^x}, \qquad t\ge 0.
\end{equation}
Note that $\widehat{L}^x_t\in [0,1]$. The occupation times formula shows that $\widehat{L}_{\zeta^{-1}(t)}^x$ is the local time at time $t$ and position $x$ of $\widehat{Y}$; see \eqref{eq:local time Yhat}. Let $\widehat{\tau}_v(x):=\inf\{t>0:\widehat{L}_t^x>v\}$ for $v\in[0,1)$. It was noted in \cite[(1.16)--(1.18)]{hu-warren} that the process $\big(\widehat{L}_{\widehat{\tau}_v(x)}^{x+h}: h\ge 0\big)$ has, in view of  \eqref{eq:generator L_tau HW}, generator 
\begin{equation}\label{jacobi}
2\ell(1-\ell)\frac{\dd^2}{\dd \ell^2} + 2 \bigg( \ind_{[0,x^-]}(h) \cdot (1-\ell) -  (1+\delta) \ell \bigg) \frac{\dd}{\dd \ell},
\end{equation}
i.e. it is a Jacobi type process. As we will discuss in the next paragraph, this result can be explained from the connection between $\widehat{Y}$ and the law of a Brownian motion conditioned on its occupation field \cite{warren2005stochastic}.
 Similarly, equation \eqref{eq:SDE L_tau drifted} becomes, on the event $\{\widehat{\tau}_v(x)<\infty\}$,
 \begin{equation}\label{eq:SDE Lhat drifted}
     \widehat{L}_{\widehat{\tau}_v(x)}^y = v + 2\int_x^y \cJ([0,\widehat{L}_{\widehat{\tau}_v(x)}^z],\dd z) 
     +2\int_{x}^y \big(\ind_{\{z\le 0\}}(1-\widehat{L}_{\widehat{\tau}_v(x)}^z) - (1+\delta) \widehat{L}_{\widehat{\tau}_v(x)}^z\big)\dd z, \qquad y\ge x,
 \end{equation}
 where 
 \begin{equation}\label{def:M}
     \cJ([0,\ell]\times[x,y]):= -(1-\ell)\, \cW\Big(\Big[0,\frac{\ell}{1-\ell}\Big]\times [x,y]\Big),\quad  0\le \ell< 1,\, x\le y.
 \end{equation}
Equation \eqref{eq:SDE Lhat drifted} appears in \cite{dawson-li12} in connection with generalized Fleming--Viot processes. The associated flow of solutions is called a stochastic Jacobi flow in \cite{aidekon2023stochastic}; see Section \ref{s:Jacobi}. Its appearance is in agreement with the results of \cite{aidekon2023stochastic} which indicate that the local time flow, see Definition \ref{def:flow}, of $\widehat{Y}$ is embedded in a Jacobi flow.

\paragraph{Conditioning a process on its occupation field}
The Bass--Burdzy flow is closely related to various models in probability. Bass and Burdzy  \cite{bass1999stochastic} were interested in ordinary differential equations controlled by a stochastic process, which appeared in physics and economics. Lupu, Sabot and Tarr\`es \cite{lupu2020fine}  constructed the scaling limit of the vertex reinforced jump process in dimension $1$ from the recurrent Bass--Burdzy flow. The Bass--Burdzy flow also appears in questions related to the conditioning of a real-valued process $\fM$ 
on its occupation measure up to a stopping time $T$. This question was treated independently by Warren and Yor \cite{warren1998burglar} when $\fM$ is a reflecting Brownian motion via the so-called \emph{Brownian burglar} which is independent of the occupation measure, and by Aldous \cite{Aldous1998Brownian} when $\fM$ is a Brownian excursion by exploiting the correspondence between Brownian excursions and continuum random trees.
When $\fM$ is a Brownian motion stopped at a hitting time,  Warren \cite{warren2005stochastic} showed that the conditioned version of $\fM$ can be expressed in terms of the process $Y$ defined in \eqref{def:Y} with $(\delta,\lambda)=(0,1)$. The appearance of the Jacobi processes in \eqref{jacobi} reflects the Ray--Knight theorems of $\widehat{Y}$; see \cite[Theorem 3]{warren1998burglar}. 

A similar question can be posed in the case of  a perturbed reflecting Brownian motion (PRBM), i.e. 
\begin{equation}\label{PRBM}
    \fM_t:=|B_t| - \mu L_B(t,0),\quad t\ge 0
\end{equation}
where $B$ is a Brownian motion, $L_B(t,0)$ is the local time of $B$ at time $t$ and position $0$, and  $\mu>0$ is a constant. Note that $\fM$ is distributed as a Brownian motion when $\mu=1$. The PRBM has been the subject of a vast literature \cite{aidekon2021path,aidekon2024infinite,perman1997perturbed,Werner1995remarks,yor1992some,carmona1994some,legall-yor1986,Chaumont-Doney1999,davis1999brownian}. A\"id\'ekon, Hu and Shi \cite{aidekon2023stochastic} studied the conditioning of $\fM$ on its occupation field and established a connection with the Bass--Burdzy flow, which we present now.
Let $L_{\fM}$ denote the bicontinuous version of the local times of $\fM$ and 
\begin{equation}\label{def-tau}
    \tau^0_1:=\inf\{ s >0: L_{\fM}(s,0)> 1 \}  
\end{equation}
be the inverse local time of $\fM$ at position $0$. We are interested in the conditional law of $(\fM_t:0\le t\le \tau^0_1)$ given $L_{\fM}(\tau_1^0,\cdot)$. We start with a few observations. Let $\fM^+$ denote the time-changed process of $\fM$ obtained by retaining only the time when $\fM$ is positive, and $\fM^-$ analogously when $\fM$ is negative. It is known that $\fM^+$ and $\fM^-$ are independent and that $\fM^+$ is distributed as a reflecting Brownian motion; see \cite[Proposition 2.1]{perman1997perturbed}. We deduce that the conditional law of $\fM$ can be recovered from the conditional laws of $\fM^-$ and $\fM^+$ given their respective occupation fields, by gluing excursions of $\fM^+$ and $\fM^-$ via their local time at $0$. The conditional law of a reflecting Brownian motion given its occupation field is given by \cite[Theorem 9]{warren1998burglar}. We then explain how to find the conditional law of $\fM^-$ given its occupation field.

 Let $(R_r: r\in \mathbb{R})$ be some process which will be defined later and define, for $t\ge 0$,
\begin{align}\label{eq:Z10}
Z^{(1)}_t &:= \int_0^{\fM_{A^{(1)}_t}} \frac{\dd r}{R_r}, & 
A^{(1)}_t &:= \inf\left\{ s  \ge 0 \,:\, \int_0^s \frac{\dd r}{R_{\fM_r}^2} > t\right\};\\
\label{eq:Z20}
Z^{(2)}_t &:= \int_0^{\fM_{\tau_1^0-A^{(2)}_t}} \frac{\dd r}{R_r}, &
A^{(2)}_t &:= \inf\left\{ s  \ge 0 \,:\, \int_{\tau_1^0-s}^{\tau_1^0} \frac{\dd r}{ R_{\fM_r}^2} > t\right\}.
\end{align}
From our choice of $R$, we have $\lim_{t\to\infty} A_t^{(1)}=t^*$ and $\lim_{t\to\infty} A_t^{(2)}=\tau_1^0-t^*$ where $t^*=\text{argmin}_{t\in [0,\tau_1^0]} \fM_t$. These transformations, as well as the time-change in \eqref{eq:zeta}, are generalizations of the ones appearing in \cite{warren1998burglar,warren2005stochastic}.
The two processes $Z^{(1)}$ and $Z^{(2)}$ are also called burglars in \cite{aidekon2023stochastic}.

Choose
\begin{equation}\label{def:Rtilde}
R_r=L_{\fM}(\tau_1^0,r)+\ind_{\{r\ge 0\}}\widetilde{R}_r^{(2/\mu)},
\end{equation}
where $\widetilde{R}^{(2/\mu)}$ is a squared Bessel process of dimension $2/{\mu}$ starting from $0$, independent of $\fM$. This amounts to looking at a two-sided PRBM defined for all $t\in \bR$; see \cite[equation (1.3)]{aidekon2023stochastic}. 
With this choice of $R$, it is proved in \cite[Theorem 6.6]{aidekon2023stochastic} that $Z^{(1)}$ and  $Z^{(2)}$ are independent of $R$. Inverting \eqref{eq:Z10} and \eqref{eq:Z20} gives the conditional law of $\fM$ given $R$. Note that the conditional law of $\fM^-$ given $R$ is also the conditional law of $\fM^-$ given $L_{\fM}(\tau_1^0,\cdot)$. Hence we can recover the conditional law of $\fM$ given $L_{\fM}(\tau_1^0,\cdot)$ indeed.
Moreover, with $\cL_{Z^{(2)}}(t,x)$ the local time of $Z^{(2)}$ and $t_{\max}:=\inf\{t\ge 0: \cL_{Z^{(2)}}(t,Z^{(2)}_t)=1\}$, the process $(Z^{(2)}_t:0\le t< t_{\max})$ is distributed as the process $\widehat{Y}=(\widehat{Y}_t:0\le t<\zeta(\infty))$ defined in \eqref{def:Yhat} with $(\delta,\lambda)=(\frac{1}{\mu}-1,1)$; see \cite[Theorem 7.4 \& Proposition 7.5]{aidekon2023stochastic}, which establishes the connection with the Bass--Burdzy flow. It was asked in \cite{aidekon2023stochastic}  whether the process $Z^{(1)}$ was itself connected to some other type of Bass--Burdzy flow. Our second main result shows that it is indeed the case, at least for an appropriate choice of $R$.

We now choose 
\begin{equation}\label{def:Rhat}
R_r=L_{\fM}(\tau_1^0,r)+\ind_{\{r\ge 0\}}\widetilde{R}^{(2)}_r,
\end{equation}
where $\widetilde{R}^{(2)}$ is a squared Bessel process of dimension $2$ starting from $0$, independent of $\fM$. It amounts to concatenating $\fM$ at time $\tau_1^0(\fM)$ with an independent 3-dimensional Bessel process; see \eqref{Mtilde}. For the same reason as above,  the conditional law of  $\fM^-$ is the same when conditioning on $R$ or on  $L_{\fM}(\tau_1^0,\cdot)$. 
Our second main result gives a link between $\cR$ and the process $Z^{(1)}$.

\begin{theorem}\label{thm:burglar}
    Let $\fM$ be the PRBM defined in \eqref{PRBM}, and $Z^{(1)}$ be the process defined in \eqref{eq:Z10} with $R$ given by \eqref{def:Rhat}.  
    Let $\cR_t(x)$ be the Bass--Burdzy flow \eqref{eq:bbflow} with $X=-\fM$ and $\lambda = 1$. Then the process $Z^{(1)}$ is independent of $R$ and is distributed as the process $\widehat{Y}$ in the notation \eqref{def:Yhat}.
\end{theorem}

Disintegrating a process with respect to its occupation measure was also studied in the context of loop soups.  Lupu, Sabot and Tarr\`es \cite{lupu2021inverting} gave the conditional law of a Brownian motion inside a Brownian loop soup on the real line. It was described there as inverting  the Ray--Knight identity. As noted in \cite{aidekon2023stochastic}, this can also be rephrased as a disintegration of the PRBM. In the discrete setting, Sabot and Tarr\`es \cite{sabot2016Inverting} and Lupu, Sabot and Tarr\`es \cite{lupu2019inverting} gave the conditional law of a loop soup on a cable graph. Werner \cite{wernermarkov} studied the Markov property of the loop soup and described the conditional law in terms of random currents, while A\"id\'ekon \cite{aidekon2020cluster} established a link between this conditional law and Jacobi flows. We refer to \cite{li2024inverting} for the analog of the inversion of the Ray--Knight identity on trees. Closely related work of Dremaux \cite{dremaux2026generalized} constructs and characterizes a one-parameter family of burglars that allow to recover the law of a Brownian loop-soup of any intensity on the real line conditionally on its occupation time field, and presents a connection of skewed Bass--Burdzy flows. Note that the problem of disintegrating a loop soup on a cable graph was instrumental in the recent discovery of the switching identity by Werner \cite{werner2025switching}.

\paragraph{Ray--Knight theorems for the generalized Bass--Burdzy flow}
Theorems~\ref{thm:rk drifted} and \ref{thm:burglar} are indeed particular consequences of results in a more general setting. 
Consider the Bass--Burdzy flow \eqref{eq:bbflow} where the driving process $X$  is the unique solution to the equation
\begin{align}\label{eq:X by BM}
X_t = \beta_t + \delta t +\left(1-\frac{1}{\mu}\right) M_t(X),\qquad t\ge 0,
\end{align}
where $\mu>0,\, \delta \ge 0$ and $\beta$ is a standard Brownian motion. We call the semimartingale $X$ a \emph{perturbed drifted Brownian motion}, abbreviated as PDBM. As shown in \eqref{eq:bang-bang}, the process can be constructed from the Brownian motion with drift $-\delta$ and reflected at $0$, introduced by Graversen and Shiryaev \cite{Graversen2000extension}. We abbreviate this process as ${\rm RBM}(-\delta)$.
If $\delta=0$, the process $X$ is distributed as a Brownian motion perturbed at its maximum, which is the opposite of a PRBM. In the case $\mu=1$, the process $X$ reduces to a drifted Brownian motion, and \eqref{eq:bbflow} yields the classical Bass--Burdzy flow. Analogously to the classical setting, \eqref{eq:bbflow} yields a flow of homeomorphisms (Proposition~\ref{p:bbflow}). We define $\widehat{Y}$ as in \eqref{def:Yhat}. In the current setting, the time-change $\zeta$ of \eqref{eq:zeta} maps $[0,\infty)$ to itself; see Lemma~\ref{l:zeta}. 

The following proposition gives the limit behavior of $\cR_t(x)$ as $t$ goes to infinity. 
\begin{proposition}\label{l:limit}
Let $\delta\ge 0$. For all $x\in\bR$, almost surely:
\begin{enumerate}[(i)]
    \item (Transient regime) if $\lambda=1$ or $\lambda=-1,\, \mu\delta>1$, then $\cR_t(x)\to \infty$ as $t\to \infty$.
    \item (Recurrent regime) if $\lambda=-1,\, \mu\delta\le 1$, then  $(\cR_t(x): t\ge 0)$ returns to 0 at arbitrarily large times.
\end{enumerate}
\end{proposition}
Observe that there is no bifurcating regime because $\delta\ge 0$. 
Also, note that the parameters in Theorem \ref{thm:burglar} always belong to the transient regime, so that $Z^{(1)}$ goes to $-\infty$ as $t$ goes to $\infty$.

We still denote by $L_t^x$ the semimartingale local time of the process $\cR(x)$ at time $t$ and position 0. The inverse local time $\tau_a(x)$ is defined by \eqref{eq:tau_a}.
As an analog of \cite[Propositions 1.1--1.2]{hu-warren}, we will show that $(L_t^x: t\ge 0,\, x\in\bR)$ admits a version such that $(t,x)\mapsto L_t^x$ is continuous (Proposition~\ref{p:biconti L}). The process $Y$ is still defined by \eqref{def:Y}. Proposition \ref{p:biconti L} shows that $Y$ admits a bicontinuous occupation density given by $\rho(t,x):=\kappa(L_t^x)$. 
Theorems \ref{thm:law HW} and \ref{thm:rk drifted} admit the following generalizations (set $\mu=1$ to recover them).

\begin{theorem}\label{thm:rk}
Let $\delta\ge 0,\, \lambda\in\{1,-1\}$.
Define the white noise $\cW$ via \eqref{eq:WN} and $\kappa(\ell)=\lambda(e^{\lambda\ell}-1)$. Then for every $x\in\bR$ and $a\ge 0$, almost surely, on the event $\{\tau_a(x)<\infty\}$,
\begin{align}\label{eq:SDE L_tau}
    L_{\tau_a(x)}^y = a - 2\int_x^y \cW([0,\kappa(L_{\tau_a(x)}^z)],\dd z) + \frac 2\mu (x^- - y^- ) - 2\delta \int_x^y \kappa(L_{\tau_a(x)}^z)\dd z,\qquad y\ge x.
\end{align}
In particular, conditionally on $\{\tau_a(x)<\infty\}$, the process $\big(L_{\tau_a(x)}^{x+h}: h\ge 0\big)$ is a diffusion starting at $a$ with the infinitesimal generator
\begin{align}\label{eq:generator L_tau}
    2\kappa(\ell)\frac{\dd^2}{\dd \ell^2} + 2 \bigg(\frac1\mu \ind_{[0,x^-]}(h) - \delta\, \kappa(\ell) \bigg) \frac{\dd}{\dd \ell}.
\end{align}
The process $\big(L_{\tau_a(x)}^{x+h}: h\ge x^-\big)$ is absorbed at 0.
In the transient case, almost surely, $L_\infty^x$ is finite for all $x\in\bR$, and the random variable $\exp(-L_\infty^0)$ has the distribution $\mathrm{beta}(\delta+1,\frac1\mu)$, resp. $\mathrm{beta}(\delta-\frac1\mu,\frac1\mu)$ when $\lambda=1$, resp. $\lambda=-1$. Moreover, 
\begin{align}\label{eq:SDE L_infty}
    L_{\infty}^y = L_\infty^x - 2\int_x^y \cW([0,\kappa(L_{\infty}^z)],\dd z) + \frac 2\mu (x^- - y^- ) - 2\delta \int_x^y \kappa(L_{\infty}^z)\dd z, \qquad y\ge x.
\end{align}
In particular, conditionally on $\{L_\infty^0 =b \}$, the two processes $(L_\infty^y: y\ge 0)$ and $(L_\infty^{-y}: y\ge 0)$ are independent, both start at $b$, and have infinitesimal generators respectively
\begin{align}\label{eq:generator L_infty}
    2\kappa(\ell)\frac{\dd^2}{\dd \ell^2} - 2\delta\, \kappa(\ell) \frac{\dd}{\dd \ell},\quad 2\kappa(\ell)\frac{\dd^2}{\dd \ell^2} + 2 \bigg(\frac1\mu - \delta\, \kappa(\ell) \bigg) \frac{\dd}{\dd \ell}.
\end{align}
The process $(L_\infty^y: y\ge 0)$ is absorbed at 0.
\end{theorem}

In the case $\lambda=1$, equation \eqref{eq:SDE Lhat drifted}  has the following analog (recall that $\widehat{L}_t^x= 1- e^{-L_t^x}$): on the event $\{\widehat{\tau}_v(x)<\infty\}$, almost surely, for $y\ge x$, 
 \begin{equation}\label{eq:SDE Lhat}
    \widehat{L}_{\widehat{\tau}_v(x)}^y = v + 2\int_x^y \cJ([0,\widehat{L}_{\widehat{\tau}_v(x)}^z],\dd z) 
     +2\int_{x}^y \left(\frac{1}{\mu}\ind_{\{z\le 0\}}(1-\widehat{L}_{\widehat{\tau}_v(x)}^z) - (1+\delta)\widehat{L}_{\widehat{\tau}_v(x)}^z \right)\dd z,
 \end{equation}
where $\cJ$ is defined in \eqref{def:M}. The flow of solutions of \eqref{eq:SDE Lhat} is called a $\mathrm{Jacobi}(\frac2\mu\,|\, 0,2+2\delta)$ flow; see Section \ref{s:Jacobi} for the precise definition.

In the recurrent regime, i.e. $\lambda=-1,\, \delta\le 1/\mu$, we further have the following Ray--Knight theorem. Note that in the case \eqref{drifted case}, i.e. $\mu=1$, the following theorem can be deduced from Theorem \ref{thm:rk drifted} by considering $-\cR$. This is not possible in the general case.

\begin{theorem}\label{thm:rk recur} 
Define the white noise on $[0,1]\times\bR$ via \eqref{eq:WN recur} in the recurrent regime. Then for every $x\in\bR$ and $a\ge 0$, almost surely,
\begin{equation}
\begin{aligned}\label{eq:SDE L_tau y<x}
    L_{\tau_a(x)}^y &= a + 2\int_y^x \cW([0,\kappa(L_{\tau_a(x)}^z)],\dd z)\\
    &\quad +\, 2\int_y^x \bigg[\bigg(\ind_{\{z\ge 0\}} + \bigg(1-\frac1\mu\bigg)\ind_{\{z<0\}}\bigg) + (\delta-1)\kappa(L_{\tau_a(x)}^z)\bigg]\dd z, \quad -M_{\tau_a(x)}(X)\le y\le x,
\end{aligned}
\end{equation}
and $L_{\tau_a(x)}^y \equiv 0$ for $y<-M_{\tau_a(x)}(X)$, so the process is absorbed at 0 on $(-\infty,0)$. 
Consequently, for every $x\in\bR$, the process $\big(L_{\tau_a(x)}^{x-h}:h\ge 0\big)$ starts at $a$ with infinitesimal generator
\begin{align}\label{eq:generator L_tau y<x}
 2\kappa(\ell)\frac{\dd^2}{\dd \ell^2} + 2\bigg[\bigg(\ind_{[0,x^+)}(h) + \bigg(1-\frac1\mu\bigg)\ind_{[x^+,\infty)}(h)\bigg) + (\delta-1)\kappa(\ell)\bigg] \frac{\dd}{\dd \ell},
\end{align}
and is absorbed at 0 on $[x^+,\infty)$.
\end{theorem}

\paragraph{Connection to the Jacobi flow}
The link between the Bass--Burdzy flow and the PRBM was given in \cite{aidekon2023stochastic} by a renewal argument. We  take a different route and prove Theorem \ref{thm:burglar}  as a consequence of Theorem \ref{thm:rk}, by exploiting the connection with Jacobi flows.  As in \cite[Section 2]{aidekon2023stochastic}, we define the  local time flow of a process as follows.
\begin{definition}[{\cite[Section 2]{aidekon2023stochastic}}]\label{def:flow}
Let $Z=(Z_t:t\ge 0)$ be a continuous process starting from $0$ which admits a bicontinuous version of its local times $(\cL_Z(t,x):t\ge 0,\,x \in\bR)$. Let $\tau_a^r(Z):=\inf\{t\ge 0: \cL_Z(t,r)>a\}$ for all $r\in\bR$ and $0\le a\le g(r):=\cL_Z(\infty,r)$ with $\inf\emptyset=\infty$ by convention. Then we define the \emph{forward local time flow} of $Z$ as the collection
\[
\cS = (\cS_{r,x}(a):x\ge r)_{r\in\bR,\, a\in[0,g(r)]},\quad \text{where}\quad \cS_{r,x}(a)=\cL_Z(\tau_a^r(Z),x).
\]
\end{definition}

The generalized Ray--Knight theorems  \cite{legall-yor1986,carmona1994some} characterize the law of the local time flow of the PRBM $\fM$ defined in \eqref{PRBM}. Recall the definition of $\widehat{Y}$ in \eqref{def:Yhat} and that of $Z^{(1)}$ and $Z^{(2)}$ in \eqref{eq:Z10}--\eqref{eq:Z20}. Let $R$ be as in \eqref{def:Rtilde}. By \cite[Proposition 4.3 \& Theorem 5.9]{aidekon2023stochastic}, the local time flows of $Z^{(1)}$ and $Z^{(2)}$ are independent of $R$  and are embedded in a $\mathrm{Jacobi}(2\,|\, 0,2+2\delta)$ flow with $\delta=\frac1\mu-1$. Since $(Z^{(2)}_t:0\le t<t_{\max})$ is distributed as $(\widehat{Y}_t:0\le t<\zeta(\infty))$, it explains equation \eqref{eq:SDE Lhat drifted}.  We will show that a similar embedding holds when $R$ is given by \eqref{def:Rhat}. 
Let $\cR$ now be the Bass--Burdzy flow driven by $X$ defined in \eqref{eq:X by BM} with $\delta=0$.
By applying again \cite[Theorem 5.9]{aidekon2023stochastic}, we will see that the local time flows of $Z^{(1)}$ and $Z^{(2)}$ are independent of $R$ and are embedded in a $\mathrm{Jacobi}(\frac2\mu\,|\, 0,2)$ flow. Using Theorem \ref{thm:rk}, we will show that $Z^{(1)}$ and $\widehat{Y}$ have the same local time flow in distribution.
It follows from the arguments of \cite[Section 2]{aidekon2023stochastic} that $Z^{(1)}$ and $\widehat{Y}$ have the same distribution, hence Theorem~\ref{thm:burglar}. 
A related description of the local time flows of a buglar as the restriction of a Jacobi flow below its ``primitive Eve'' is given in \cite[Section 4.4]{dremaux2026generalized}.
We sum up this discussion with the following theorem, which is a direct consequence of Proposition~\ref{p:Jacobi} and Corollary~\ref{c:time-change}. 

\begin{theorem}\label{thm:Jacobi}
Let $S=(S_{r,x}(v):x\ge r)_{r\in\bR,\, v\in[0,1]}$ be a $\mathrm{Jacobi}(\frac2\mu\,|\, 0,2+2\delta)$ flow. Then $\Psi_x:=\lim_{r\to-\infty} S_{r,x}(0)$ exists for every $x\in\bR$, and $S_{r,x}(\Psi_r)=\Psi_x$ for all $x\ge r$.
The forward local time flow of the process $\widehat{Y}$ has the same law as $(S_{r,x}(v):x\ge r)_{r\in\bR,\, v\in[0,\Psi_r]}$. 
\end{theorem}

\paragraph{Organization}
The organization of this paper is as follows. Section~\ref{s:bbflow} collects fundamental results on the stochastic flow induced by \eqref{eq:bbflow}, while Section~\ref{s:WN} provides a brief overview of the stochastic integral with respect to a white noise following \cite{Walsh1986AnIT}. We will follow the idea of \cite{aidekon2024infinite} and introduce the excursion filtrations and their enlargement in Section~\ref{s:excursion}. Section~\ref{s:integral} is devoted to the definition of the white noise $\cW$ appearing in Theorems~\ref{thm:rk} and~\ref{thm:rk recur} and the representation of the stochastic integral with respect to $\cW$, which will then be applied to prove Theorems~\ref{thm:rk} and~\ref{thm:rk recur}. Section~\ref{s:Jacobi} establishes the connection between the Bass--Burdzy flow driven by a PDBM and the Jacobi flow, which in turn yields a connection with the burglar as shown in Section~\ref{s:burglar}. Finally, we prove the differentiability in $x$ of the Bass--Burdzy flow \eqref{eq:bbflow} driven by a PDBM in Section~\ref{s:diffeomorphism}.

\paragraph{Acknowledgment} We are grateful to Elie A\"id\'ekon for fruitful discussions and helpful suggestions.  

\paragraph{AI declaration} We used AI tools only for language polishing and correction of typographical errors during the final stage of manuscript preparation.

\paragraph{Notation}
For a process $Z=(Z_t:t\ge 0)$, define 
$$
I_t(Z):=\inf_{0\le s\le t} Z_s,\quad M_t(Z):=\sup_{0\le s\le t} Z_s,\quad T_x(Z):=\inf\{t\ge 0: Z_t=x\},
$$
with $\inf\emptyset = \infty$ by convention. 
If $Z$ is a continuous semimartingale, denote by $L_Z(t,x)$ the local time of $Z$, i.e. for every non-negative measurable function $f$ and $t\ge 0$, we have the identity
\begin{align}\label{eq:occupation times formula}
\int_0^t f(Z_s)\dd \langle Z\rangle_s = \int_\bR f(x) L_Z(t,x) \dd x,
\end{align}
where $\langle Z\rangle$ represents the quadratic variation process of $Z$.
We will always work with the version such that $L_Z(t,x)$ is continuous in $t$ and c\`adl\`ag in $x$; see \cite[Theorem VI.1.7]{revuz2013continuous}. In general cases, we shall denote by $\cL_Z(t,x)$ the occupation density of $Z$ (if it exists), i.e. we replace $L_Z$ by $\cL_Z$ and $\dd\langle Z\rangle_s$ by $\dd s$ in \eqref{eq:occupation times formula}, and still call $\cL_Z$ the local time process of $Z$.


\section{Preliminaries}
In this section, we introduce some basic facts on the Bass--Burdzy flow driven by a perturbed drifted Brownian motion, which lay the foundation of Section~\ref{s:excursion}. We also recall the stochastic integral with respect to a white noise appearing in \cite{Walsh1986AnIT}, which will be applied in Section~\ref{s:integral}.

\subsection{Bass--Burdzy flow driven by a PDBM}\label{s:bbflow}
Fix $\delta\ge 0,\, \mu>0$ and let $B$ be a standard Brownian motion. We define $U=(U_t,\, t\ge 0)$ to be the unique strong solution to the equation
\begin{align}\label{eq:U}
    U_t = B_t - \delta \int_0^t \text{sgn}(U_s)\dd s.
\end{align}
The process $U$ is recurrent. Let $B^\delta_t := B_t+\delta t$. By \cite[Theorem 1]{Graversen2000extension}, we have 
\begin{align}\label{eq:U-B}
    \big( |U|,L_U(\cdot,0) \big) \eqdis \big( M(B^\delta)-B^\delta,M(B^\delta) \big).
\end{align}
The process $|U|$ is distributed as the \emph{reflecting Brownian motion with drift $-\delta$} starting at 0, also abbreviated as $\rm RBM(-\delta)$; see \cite[Theorem 2]{Graversen2000extension}. 
As an analog of PRBM, we define the process
\begin{align}\label{eq:bang-bang}
    X_t: = \mu L_U(t,0) - |U_t|, \qquad t\ge 0.
\end{align}
We observe that $M_t(X) = \mu L_U(t,0)$, and that $(X,M(X))$ is a strong Markov process. Moreover, by Tanaka's formula \cite[Theorem VI.1.2]{revuz2013continuous}, $X$ satisfies \eqref{eq:X by BM} where $\beta_t:= -\int_0^t {\rm sgn}(U_s)\dd B_s$. The process $\beta$ is a Brownian motion by P. L\'evy's characterization theorem (cf. \cite[Theorem IV.3.6]{revuz2013continuous}).
Therefore,  $X$ is a perturbed drifted Brownian motion as introduced in Section~\ref{s:introduction}. 

\medskip

We first introduce some properties of $X$. For $x\in\bR$, set 
\begin{align}\label{eq:Tx}
\cT_x:=\inf\{t\ge 0:X_t\ge -x\}.
\end{align}
Then $\cT_x=0$ for $x\ge 0$, while $\cT_x=T_{-x}(X)=T_{-x}(M(X))$ for $x<0$. 
\begin{lemma}\label{l:properties X} 
Suppose $\delta\ge 0$ and $x\in\bR$. The following statements hold almost surely:
\begin{enumerate}[(i)]
    \item $X_t/t\to \mu\delta$ as $t\to\infty$, and $X$ is recurrent if $\delta=0$;
    \item $M_t(X)>x^-$ for all $t>\cT_x$;
    \item $M_t(X)\to \infty$ as $t\to \infty$, and $\mathrm{supp\ } \dd M(X) = \{t\ge 0: X_t = M_t(X)\}$.
\end{enumerate}
\end{lemma}
\begin{proof}
    From \eqref{eq:U-B}, $M(X)=\mu L_U(\cdot,0)\eqdis \mu M(B^\delta)$ and $|U|\eqdis M(B^\delta)-B^\delta$. Hence $M_t(X)\to \infty$ as $t\to \infty$. Furthermore, we have $M_t(X)/t\to \mu\delta$ and $|U_t|/t\to 0$ as $t\to \infty$. 
    It follows from \eqref{eq:bang-bang} that $X_t/t\to \mu\delta$. In the case $\delta=0$, $X$ has the law of the opposite of a perturbed reflecting Brownian motion, hence is recurrent. This completes the proof of (i).
    From \eqref{eq:U-B}, for any fixed $y>0$, almost surely $M_t(X)$ will immediately increase after time 0 and after the time when it reaches $y$. We deduce (ii) by the definition of $\cT_x$.
    Finally, the second statement in (iii) is equivalent to the equality $\mathrm{supp\ } \dd L_U(\cdot,0) = \{t\ge 0: U_t = 0\}$. Recall that $U$ is the unique strong solution to \eqref{eq:U}. The result holds when replacing $U$ by $B$ by \cite[Proposition VI.2.5]{revuz2013continuous}, and therefore also holds for $U$ by applying Girsanov's theorem \cite[Theorem VIII.1.12]{revuz2013continuous}. 
\end{proof}

Recall that $\lambda\in\{-1,1\}$. For each $x\in\bR$, we are interested in the solution $(\cR_t(x):t\ge 0)$ to equation \eqref{eq:bbflow} with driving function $X$ given by \eqref{eq:bang-bang}.

\begin{proposition}\label{p:bbflow}
There exists a version of $(\cR_t(x):t\ge 0,\, x\in\bR)$ such that each process $(\cR_t(x):t\ge 0)$ gives a solution to \eqref{eq:bbflow}, and that $x\mapsto \cR_t(x)$ is continuous and strictly increasing for every $t\ge 0$.
\end{proposition}

\begin{proof}
Define $s(x)=\int_0^x e^{-2\lambda\max\{y,0\}-2\delta y}\dd y$. Then by It\^o's formula, $\cX_t(y):= s(\cR_t(s^{-1}(y)))$ satisfies the equation
\begin{align}\label{eq:change bbflow}
    \cX_t(y) = y + \int_0^t \sigma\big(\cX_s(y)\big) \dd \widetilde{X}_s, \qquad t\ge 0,
\end{align}
where for $x\in s(\bR)$,
\begin{align}\label{eq:sigma & X}
    \sigma(x) = s'(s^{-1}(x)) = 1 - 2\delta x - 2\lambda \max\{x,0\},\qquad \widetilde{X}_t = \beta_t + (\mu-1)L_U(t,0) = X_t -\delta t.
\end{align}
Hence $\tX$ is a continuous semimartingale. We extend $\sigma$ by \eqref{eq:sigma & X} outside $s(\bR)$ whenever $s(\bR)\ne \bR$, and thus obtain a globally Lipschitz continuous function $\sigma:\bR\to\bR$. By \cite[Theorem 5]{uppman1982}, there exists a version of $(\cX_t(y):t\ge 0,\, y\in\bR)$ such that each $\cX_t(y)$ solves \eqref{eq:change bbflow} and that the map $y\mapsto \cX_t(y)$ is a homeomorphism from $\bR$ to $\bR$ for every $t\ge 0$. Moreover, $\cX(y)\equiv y$ for $y\in \partial (s(\bR))$ whenever $s(\bR)\ne \bR$ from the construction of $\sigma$, and thus $\cX(y)\in s(\bR)$ for $y\in s(\bR)$. The proposition then follows since $s\in C^1(\bR)$ is strictly increasing.
\end{proof}

The following lemma will be useful when applying Tanaka's formula to $\cR_t(x)$ and analyzing the limit behavior. 

\begin{lemma}\label{l:inclusion}
Fix $x\in\bR$. Almost surely, the measure $\dd M_t(X)$ restricted to the set $[\cT_x,\infty)$ is supported on $\{t\ge \cT_x:\cR_t(x)\ge 0\}$.
Moreover, 
\begin{align}\label{eq:dI 0}
    \int_0^\infty \ind_{\{\cR_s(x)=0\}} \dd M_s(X) = 0.
\end{align}
In particular, for each $x\in\bR$, almost surely for every $t\ge 0$,
\begin{align}\label{eq:dI}
    \int_0^t \ind_{\{\cR_s(x)\le 0\}} \dd M_s(X) = M_t(X)\wedge x^-, \quad
    \int_0^t \ind_{\{\cR_s(x)>0\}} \dd M_s(X) = (M_t(X) - x^-)\vee 0.
\end{align}
\end{lemma}

\begin{proof}
We have $\cR_t(x) = x + X_t$ for $0\le t\le \cT_x$ therefore $\cR_{\cT_x}(x)=x^+$. Let $s>\cT_x$ be such that $\cR_s(x)<0$ and $g_s:=\sup\{u<s:\cR_u(x)=0\} \in[\cT_x,s)$. From \eqref{eq:bbflow}, we derive that $\cR_s(x) = X_s-X_{g_s}<0$ hence $s\notin \mathrm{supp}~\dd M_t(X)$. 

We now prove \eqref{eq:dI 0}.
Suppose first $\lambda=1$. 
By Lemma~\ref{l:properties X}~(ii)--(iii), $\cR_t(x)\ge x+X_t= x + M_t(X)>0$ for all $t\in\text{supp } \dd M_s(X)$ with $t>\cT_x$, which implies that the intersection of $\mathrm{supp}~\dd M_t(X)$ and $\{t:\cR_t(x)=0\}$ is the singleton $\{\cT_x\}$ if $x\le 0$ or empty if $x>0$, and hence \eqref{eq:dI 0} follows. 
Suppose now $\lambda=-1$. Let us prove that $\xi:=\inf\{t>\cT_x: X_t=M_t(X),\, \cR_t(x)=0\}$ is strictly greater than $\cT_x$. 
By the strong Markov property of $X$ at time $\cT_x$, the two processes $X_{\cdot+\cT_x} - x^-$ and $X$ have the same law. Therefore almost surely, there exists some (random) $\varepsilon>0$ such that
\[
X_s-x^- = M_s(X) - M_{\cT_x}(X)>2(s-\cT_x),\qquad \forall\, s\in(\cT_x,\cT_x+\varepsilon)\cap \mathrm{supp}~\dd M_t(X),
\]
where the first equality follows from Lemma~\ref{l:properties X}~(iii), and the second inequality follows by a combination of \eqref{eq:U-B} and law of iterated logarithm for the Brownian motion.
On the other hand, for any $s>\cT_x$ with $\cR_s(x)=0$, equation \eqref{eq:bbflow} implies that $X_s-x^- = X_s-X_{\cT_x} = -x^+ + \int_{\cT_x}^s \ind_{\{\cR_u(x)>0\}}\dd u \le s-\cT_x$. Combining these two facts, we obtain that $\xi\ge \cT_x+\varepsilon$. Applying the strong Markov property of $X$ at time $\xi$ and repeating the above argument, we deduce that the intersection of $\mathrm{supp}~\dd M_t(X)$ and $\{t:\cR_t(x)=0\}$ has at most countably many elements, which leads to \eqref{eq:dI 0}. 
Equation~\eqref{eq:dI} is now a combination of \eqref{eq:dI 0} and the fact that $s\notin \mathrm{supp}~\dd M_t(X)$ for all $s>\cT_x$ with $\cR_s(x)<0$.
\end{proof}

We now prove Proposition~\ref{l:limit}.

\begin{proof}[Proof of Proposition~\ref{l:limit}]

We have $T_y(X)<\infty$ for all $y\ge 0$ by Lemma~\ref{l:properties X}~(iii). 
 Suppose first $\lambda=1$. For $x\in\bR$, we take any $y>0$ and derive from \eqref{eq:bbflow} that
\[
\cR_{T_y(X)+t}(x) =  \cR_{T_y(X)}(x) + (X_{T_y(X)+t}-y) + \int_0^t \ind_{\{\cR_{T_y(X)+s}(x)> 0\}}\dd s.
\]
By the strong Markov property of $X$ at time $T_y(X)$, we obtain $\cR_{T_y(X)+\cdot}(x)\eqdis \widetilde{\cR}\big(\cR_{T_y(X)}(x)\big)$ where $\widetilde{\cR}$ is an independent copy of $\cR$. The fact that $\cR_{T_y(X)}(x)\ge x+y$ and the increasing property of $x\mapsto\cR_t(x)$ imply that the value of $\bP(\lim_{t\to \infty}\cR_t(x)=\infty)$ does not depend on $x$. 
On the other hand, by Lemma~\ref{l:properties X}~(i), we have $X_t+t\to \infty$ as $t\to \infty$ and therefore there exists some finite (random) $\mathrm{x}>0$ large enough such that
\[
\cR_t(\mathrm{x}) = \mathrm{x} + X_t +t >0,\qquad  \forall\, t\ge 0,
\]
so that $\cR_t(\mathrm{x})\to \infty$ as $t\to \infty$. Hence $\bP(\lim_{t\to \infty}\cR_t(x)=\infty)\ge \bP(\mathrm{x}\le x)$.
Sending $x\to\infty$ yields the result.

Suppose now $\lambda=-1$. Then we have $\cR_t(x)\ge x+X_t-t$ hence $\liminf_{t\to\infty}\cR_t(x)/t \ge \mu\delta-1$ by Lemma~\ref{l:properties X}~(i). It follows that $\cR_t(x)\to \infty$ in the case $\mu\delta> 1$. Suppose henceforth $\mu\delta\le 1$. By Lemma~\ref{l:inclusion}, there exists $T\ge 0$ arbitrarily large such that $\cR_T(x)>0$. Suppose, for the sake of contradiction, that $\cR_t(x)>0$ for all $t\ge T$. Then from \eqref{eq:bbflow}, we have
\[
0<\cR_t(x) = \cR_T(x) + (X_t-X_T) - (t-T),\quad \forall\, t\ge T,
\]
which contradicts Lemma~\ref{l:properties X}~(i) in the case $\mu\delta<1$. It remains to deal with the case $\mu\delta=1$ (hence necessarily $\delta>0$). We rewrite 
\[
X_t -t = [\mu (L_U(t,0)-|U_t|) - t] + (\mu -1)|U_t|.
\]
From \eqref{eq:U-B}, the first term on the right-hand side has the law of $\mu B$.
Furthermore, by the description of $|U|$ in \cite[Theorem 2]{Graversen2000extension}, it is standard to show that $|U|$ has scale function $s_\delta(x) = \exp(2\delta x)$ and is positive recurrent. It follows by applying Motoo's method (cf. \cite[Chapter 4.12]{ito2012diffusion}) that $\limsup_{t\to\infty} |U_t|/\log t <\infty$. Together with the law of the iterated logarithm for Brownian motion, we derive that $\liminf_{t\to \infty} (X_t-t)/\sqrt{t} = -\infty$, which again contradicts the assumption that $\cR_t(x)>0$ for all $t\ge T$. The proof is thus complete.
\end{proof}

From \eqref{eq:dI} and Tanaka's formula, we deduce that
\begin{align}
    &\cR_t^+(x) = x^+ + \int_0^t \ind_{\{\cR_s(x)>0\}} \dd X_s + \lambda \int_0^t \ind_{\{\cR_s(x)>0\}} \dd s + \frac12 L_t^x\notag\\
    =\ & x^+ + \int_0^t \ind_{\{\cR_s(x)>0\}} \dd \beta_s + (\delta+\lambda) \int_0^t \ind_{\{\cR_s(x)>0\}} \dd s +\frac{\mu-1}{\mu} [(M_t(X) - x^-)\vee 0] + \frac12 L_t^x, \label{eq:Tanaka+}
    \end{align}
    and
    \begin{align}
    &\cR_t^-(x) = x^- - \int_0^t \ind_{\{\cR_s(x)\le 0\}} \dd X_s + \frac12 L_t^x\notag\\
    =\ & x^- - \int_0^t \ind_{\{\cR_s(x)\le 0\}} \dd \beta_s -\delta \int_0^t \ind_{\{\cR_s(x)\le 0\}} \dd s - \frac{\mu-1}{\mu} (M_t(X)\wedge x^-) + \frac12 L_t^x. \label{eq:Tanaka-}
\end{align}
They will play a crucial role in characterizing the law of $(L_{\tau_a(x)}^y:y\ge x)$ as shown in Sections~\ref{s:excursion} and \ref{s:integral}.

\medskip

From \eqref{eq:bbflow}, $\cR_t(x)$ is almost surely a semimartingale for each $x\in\bR$. Hence it admits a version of the local time process $(L_{\cR(x)}(t,y):t\ge 0,\, y\in\bR)$ which is almost surely continuous in $t$ and c\`adl\`ag in $y$. 

\begin{lemma}\label{l:L conti}
For each $x\in\bR$, almost surely, $z\mapsto L_{\cR(x)}(t,z)$ is continuous at 0 for every $t\ge 0$.    
\end{lemma}
\begin{proof}
Recall that $\cX_t(y)=s(\cR_t(s^{-1}(y)))$ in the proof of Proposition~\ref{p:bbflow}. By the occupation times formula \eqref{eq:occupation times formula},
\[
L_{\cR(x)} (t,z) = \frac1{s'(z)} L_{\cX(s(x))}(t,s(z)).
\]
Notice that $s\in C^1(\bR)$, $s(0)=0$ and $s'>0$. By \cite[Theorem VI.1.7]{revuz2013continuous}, it suffices to prove that almost surely, $\int_0^t \ind_{\{\cX_s(y)=0\}} \dd M_s(X) = 0$ for all $t\ge 0$. Since $\cX_t(y)=0$ if and only if $\cR_t(s^{-1}(y))=0$, the desired equality follows from \eqref{eq:dI 0}.
\end{proof}

Recall that $L_t^x=L_{\cR(x)}(t,0),\, \kappa(\ell)=\lambda(e^{\lambda\ell}-1)$ and that $Y_t=\cR^{-1}_t(0)$ is the unique $x$ satisfying $\cR_t(x)=0$. 
The following result is an analog of \cite[Propositions 1.1--1.2]{hu-warren}.
\begin{proposition}\label{p:biconti L}
\begin{enumerate}[(i)]
    \item There exists a version of $(L_t^x: t\ge 0,\, x\in\bR)$ such that $(t,x)\mapsto L_t^x$ is continuous. Moreover,
    \begin{align}\label{eq:derivative}
    \frac{\partial \cR_t(x)}{\partial x} = \exp(\lambda L_t^x).
    \end{align}

    \item The process $Y$ admits a bicontinuous occupation density $\rho(t,x) := \kappa(L_t^x)$, i.e. for every bounded measurable function $f$,
    \begin{align}\label{eq:occupied}
        \int_0^t f(Y_s)\dd s = \int_\bR f(x)\rho(t,x)\, \dd x.
    \end{align}
\end{enumerate}
\end{proposition}

\begin{proof}
Recall the definition of $\widetilde{X}$ in \eqref{eq:sigma & X}. Define $Z_t(y):=\int_0^t \sigma'\big(\cX_s(y)\big)\ind_{\{\cX_s(y)\ne 0\}}\dd \tX_s$ for $t\ge 0,\, y\in s(\bR)$. By Proposition~\ref{p:derivative} and \eqref{eq:dI 0}, we have for all $y\in s(\bR)$,
\begin{align*}
&\frac{\partial \cX_t(y)}{\partial y} = \exp\Big(Z_t(y) - \frac12 \langle Z_\cdot(y)\rangle_t\Big)\\
=\, & \exp\bigg(-\int_0^t [2\delta\ind_{\{\cX_s(y)\le 0\}} + 2(\delta+\lambda)\ind_{\{\cX_s(y)> 0\}}] \dd \tX_s - \int_0^t [2\delta^2\ind_{\{\cX_s(y)\le 0\}} + 2(\delta+\lambda)^2\ind_{\{\cX_s(y)> 0\}}]\dd s\bigg)\\
=\, &\exp\bigg(-2\delta\int_0^t \ind_{\{\cX_s(y)\le 0\}}\dd X_s - 2(\delta+\lambda)\left(\int_0^t \ind_{\{\cX_s(y)> 0\}}\dd X_s+\lambda\int_0^t \ind_{\{\cX_s(y)> 0\}}\dd s\right)\bigg),
\end{align*}
where the last equality holds since $X_t=\widetilde{X}_t+\delta t$ from \eqref{eq:sigma & X}.
Since $\cR_t(x)\le 0$ if and only if $\cX_t(s(x))\le 0$, it follows in view of \eqref{eq:Tanaka+}--\eqref{eq:Tanaka-} that
\begin{align}\label{eq:derivative cX-Tanaka}
\frac{\partial \cX_t(y)}{\partial y}\bigg|_{y=s(x)} = \exp\big(\lambda L_t^x + 2\delta(\cR_t^-(x)-x^-)  - 2(\delta+\lambda)(\cR_t^+(x)-x^+)\big).
\end{align}
Having \eqref{eq:Tanaka+} and \eqref{eq:Tanaka-}, we can follow the idea from the proof of \cite[Propositions 1.1--1.2]{hu-warren} and deduce the result. 
Indeed, by Proposition~\ref{p:derivative}, $\partial \cX_t(x)/\partial x$ admits a version such that it is bicontinuous in $(t,x)$. Therefore the map $x\mapsto \cX_t(x)$ is locally Lipschitz continuous, which implies the continuity of $(t,x)\mapsto \cX_t(x)$. Since $\cR_t(x) = s^{-1}\big(\cX_t(s(x))\big)$, the map $(t,x)\mapsto \cR_t(x)$ is also continuous. 
Equation \eqref{eq:derivative cX-Tanaka} then implies the bicontinuity of  $L_t^x$, and equation~\eqref{eq:derivative} follows by a combination of \eqref{eq:derivative cX-Tanaka} and the identity $s'(x)= e^{-2\lambda\max\{x,0\}-2\delta x}$.

To prove the second result, we apply \eqref{eq:derivative} and the monotonicity of $x\mapsto\cR_t(x)$ and thus derive that, for all $t\ge 0$ and $x_1< x_2$,
\begin{align*}
&\int_{x_1}^{x_2} e^{\lambda L_t^x}\dd x = \cR_t(x_2)-\cR_t(x_1)\\
=\, & (x_2-x_1) + \lambda\int_0^t \ind_{\{\cR_s(x_1)\le 0<\cR_s(x_2)\}}\dd s = (x_2-x_1) + \lambda\int_0^t\ind_{\{x_1\le Y_s<x_2\}} \dd s.
\end{align*}
Therefore $\rho(t,x)=\kappa(L_t^x)$ is a version of the local time of $Y$.
\end{proof}


\subsection{Stochastic integral with respect to a white noise}\label{s:WN}
In this section, we refer to \cite[Chapter 2]{Walsh1986AnIT} and introduce the concept of the stochastic integral with respect to a white noise. Let $\cB(\bR_+)$ be the set of all Borel subsets of $\bR_+$, and for $A\in\cB(\bR_+)$, we denote by $|A|$ the Lebesgue measure of $A$. Let $\cL\subseteq\cB(\bR_+)$ be the set of Borel subsets of $\bR_+$ with finite Lebesgue measure and $\rF=(\rF_x:x\in\bR)$ be a filtration satisfying the usual conditions.

A white noise $\cW$ in $\bR_+\times \bR$ is a collection of random variables $\cW(C)$ for $C$ a Borel subset of $\bR_+\times\bR$ with finite Lebesgue measure, such that for each $C$, $\cW(C)$ is a centered Gaussian variable with variance $|C|$, that $\cW(C)$ and $\cW(C')$ are independent for every $C$ and $C'$ disjoint, and that $C\mapsto \cW(C)$ is countably additive in $L^2$.
We fix some $r\in\bR$ and define
\begin{align}\label{eq:mart_WN}
\cM_x(A):=\cW(A\times [r,x]),\qquad A\in\cL,\ x\ge r.  
\end{align}
Then $A\mapsto\cM_x(A)$ is countably additive. We call $\big(\cM_x(A):x\ge r,\, A\in\cL)$ an $\rF$-\emph{martingale measure} if additionally $\big(\cM_x(A):x\ge r)$ is an $\rF$-martingale for all $A\in\cL$.
We call $f$ an \emph{elementary function} if it is of the form $f=f(\ell,x)=Z\ind_{(x_1,x_2]}(x)\ind_A(\ell)$ where $r<x_1<x_2$, $A\in\cL$ and $Z$ is a bounded $\rF_{x_1}$-measurable random variable, and define the stochastic integral of such $f$ with respect to $\cM$ as 
\[
f\cdot \cM := Z\big(\cM_{x_2}(A)-\cM_{x_1}(A)\big).
\]
We extend the stochastic integral to every \emph{simple function}, i.e. a finite sum of elementary functions, by linearity. Let $\rS$ be the set of simple random functions and $\rL^2$ be the set of $\rF$-predictable random functions $f$ such that
\[
\Vert f\Vert _{\rL^2}^2 := \bE\left[ \int_{\bR_+\times[r,\infty)} f^2(\ell,x) \dd\ell\dd x \right] < \infty.
\]
Then $\rS$ is dense in $\rL^2$ under the norm $\Vert \cdot\Vert _{\rL^2}$; see \cite[Proposition 2.3]{Walsh1986AnIT}. Thus we define $f\cdot \cM$ for $f\in\rL^2$ by isometry, and write
\[
f\cdot \cM  = \int_{\bR_+\times[r,\infty)} f(\ell,x) \cW(\dd\ell, \dd x).
\]
By \cite[Theorem 2.5]{Walsh1986AnIT}, for every $f\in\rL^2$, the process $(\int_{\bR_+\times[r,y]}f(\ell,x)\cW(\dd \ell,\dd x): y\ge r)$ is a continuous $\rF$-adapted $L^2$-martingale with quadratic variation process $(\int_{\bR_+\times[r,y]}f^2(\ell,x)\dd\ell\dd x: y\ge r)$.

As for stochastic integration with respect to continuous local martingales, we can also define $f\cdot \cM$ for $f\in \rL^2_{loc}$, i.e. all $\rF$-predictable functions $f(\ell,x)$ such that $\int_r^y\int_{\bR_+} f^2(\ell,x)\dd \ell\dd x$ is a.s. finite for all $y\ge r$. We introduce an increasing sequence of $\rF$-stopping times $\{{\bf x}_n:n\ge 1\}$ as
\begin{align}\label{eq:stopping}
    {\bf x}_n = \inf\left\{y\ge r: \int_{r}^y\int_{\bR_+} f^2(\ell,x)\dd \ell\dd x>n\right\}.
\end{align}
Then $f_n(\ell,x):=f(\ell,x)\ind_{\{x<{\bf x}_n\}}\in\rL^2$, and hence we can define for $y\ge r$,
\begin{equation*}
    \int_{\bR_+\times[r,y]}f(\ell,x)\cW(\dd \ell,\dd x):=\int_{\bR_+\times[r,y]}f_n(\ell,x)\cW(\dd \ell,\dd x) \quad  \text{on the event } \{r\le y<{\bf x}_n\},
\end{equation*}
which, as $y$ varies, is a continuous $\rF$-local martingale with quadratic variation $(\int_{\bR_+\times[r,y]}f^2(\ell,x)\dd \ell\dd x: y\ge r)$.
For $f$ of the form $f(\ell,y)=\xi_y\ind_{[0,\sigma_y]}(\ell)$ with $\sigma\ge 0$, we have $f\in\rL^2_{loc}$ if the two processes $(\sigma_y:y\ge r)$ and $(\xi_y:y\ge r)$ are continuous and $\rF$-adapted. In this case, we write the stochastic integral as $\int_{r}^y \xi_x \cW([0,\sigma_x],\dd x)$, which is a continuous $\rF$-local martingale with quadratic variation process $(\int_{r}^y\sigma_x\xi_x^2\dd x: y\ge r)$.

\section{Excursions Related to a Solution}\label{s:excursion}
Recall the process $(\cR_t(x):t\ge 0)$ in \eqref{eq:bbflow} for $x\in\bR$. Let $\rB=(\rB_t:t\ge 0)$ be the right-continuous filtration generated by the Brownian motion $\beta$. We introduce the following notation:
\begin{equation}\label{eq:A+}
\begin{aligned}
     A_t^{+,x} &:= \int_0^t\ind_{\{\cR_s(x)>0\}}\dd s,&\alpha_t^{+,x} := \inf\{s>0: A_s^{+,x}>t\},\\
     X^{+,x}_t&:=X_{\alpha^{+,x}_t}, &\cR_t^{+,x}(y) := \cR_{\alpha_t^{+,x}}(y) \quad \text{for }  t< A^{+,x}_\infty, 
\end{aligned}
\end{equation}
and $L^{+,x}_t:=L_{\cR^{+,x}(x)}(t,0)$. We define $A_t^{-,x},\, \alpha_t^{-,x}$ and $\cR_t^{-,x}(y)$ similarly but considering the time during which $\cR_t(x)\le 0$, and set $L^{-,x}_t:= L_{-\cR^{-,x}(x)}(t,0)$. 
Recall the definition of $\cT_x$ in \eqref{eq:Tx}. We have $\alpha^{+,x}_0=\cT_x$ so that $X^{+,x}_0=x^-$. We observe that $A^{+,x}_\infty=\infty$ for all $x\in\bR$ by Proposition~\ref{l:limit}, and that $A^{-,x}_\infty=\infty$ in the recurrent case while $A^{-,x}_\infty<\infty$ in the transient case.
We define $\rF^{+,x}_t := \sigma(\cR^{+,x}_s(x):0\le s\le t)$ and 
\begin{align}\label{eq:beta+}
    \beta^{+,x}_t = \int_0^{\alpha_t^{+,x}}\ind_{\{\cR_s(x)> 0\}} \dd \beta_s,\quad t\ge 0.
\end{align}
Then $\beta^{+,x}$ is a $\rB_{\alpha^{+,x}}$-Brownian motion by \cite[Theorem V.1.6]{revuz2013continuous}, hence also an $\rF^{+,x}$-Brownian motion.
We define $\beta^{-,x}$ in the recurrent case via \eqref{eq:beta+} with $(\alpha^{+,x},\ind_{\{\cR_s(x)> 0\}})$ replaced by $(\alpha^{-,x},\ind_{\{\cR_s(x)\le 0\}})$. In the transient case, we apply the method of enlargement of filtration: we introduce another Brownian motion $(\widetilde{B}_t,t\geq 0)$ independent of $B$, and define 
\begin{align}\label{eq:beta-}
    \beta^{-,x}_t &=\begin{cases}
        \int_0^{\alpha_t^{-,x}}\ind_{\{\cR_s(x)\le 0\}}\dd \beta_s, & 0\le t<A_{\infty}^{-,x},\\
        \int_0^{\infty}\ind_{\{\cR_s(x)\le 0\}}\dd \beta_s+\widetilde{B}_{t-A_{\infty}^{-,x}},\quad &  t\geq A_{\infty}^{-,x}.
    \end{cases}
\end{align}
In both cases, $\beta^{-,x}$ is a Brownian motion independent of $\beta^{+,x}$ by Knight's theorem \cite[Theorem V.1.10]{revuz2013continuous}.

\begin{proposition}\label{p:R- delta>0}
Assume $\delta\ge 0$ and $x\in\bR$. Then on the event $\{A^{-,x}_\infty>0\}$,
    \begin{align}\label{eq:R- delta>0}
       -\cR^{-,x}_t(x) =  \frac1\mu x^- - \beta^{-,x}_t - \delta t + \left(1-\frac1\mu\right) I_t(-\cR^{-,x}(x)) + \frac12 L^{-,x}_t,\quad 0\le t< A^{-,x}_\infty.
    \end{align}
\end{proposition}

\begin{proof}
Replacing $t$ by $\alpha_t^{-,x}$ in \eqref{eq:Tanaka-}, we obtain that
\begin{align}
    -\cR^{-,x}_t(x) = \cR_{\alpha_t^{-,x}}^-(x) 
    = x^- - \beta^{-,x}_t - \delta t - \left(1-\frac1\mu\right)(M_{\alpha^{-,x}_t}(X)\wedge x^-) + \frac12 L_{\alpha_t^{-,x}}^x . \label{eq:R-' delta>0}
\end{align}
Observe that $L^x_{\alpha^{-,x}_t} = L_{\cR(x)}(\alpha^{-,x}_t,0) = L^{-,x}_t$. 
Indeed, by Lemma~\ref{l:L conti},
\[
L_{\alpha^{-,x}_t}^x = \lim_{\varepsilon\downarrow 0} \frac1\varepsilon \int_0^{\alpha^{-,x}_t} \ind_{\{\cR_s(x)\in(-\varepsilon,0)\}}\dd s = \lim_{\varepsilon\downarrow 0} \frac1\varepsilon \int_0^t \ind_{\big\{\cR_{\alpha^{-,x}_s}(x)\in(-\varepsilon,0)\big\}}\dd s = L_{-\cR^{-,x}(x)}(t,0).
\]
In the case $x\ge 0$, we have $M_{\alpha^{-,x}_t}(X)\wedge x^- = I_t( -\cR^{-,x}(x)) =0$ for all $0\le t< A^{-,x}_\infty$. To treat the case $x<0$, we observe that $\alpha^{-,x}_t=t$ for $0\le t< T_{-x}(X)$ and thus $-x>M_t(X)=M_t(\cR(x))-x = -I_t(-\cR^{-,x}(x))-x$, while $M_{\alpha^{-,x}_t}(X)\wedge x^- = x^-$ and $I_t( -\cR^{-,x}(x)) =0$ for $T_{-x}(X)\le t< A^{-,x}_\infty$. We deduce that for all $x\in\bR$ and $0\le t< A^{-,x}_\infty$,
\[
M_{\alpha^{-,x}_t}(X)\wedge x^- = x^- - I_t(-\cR^{-,x}(x)).
\]
Hence \eqref{eq:R- delta>0} follows from \eqref{eq:R-' delta>0}.
\end{proof}

In the transient case, we have $\cR^{-,x}_{A^{-,x}_\infty}(x) = 0$. Recall the definition of $\beta^{-,x}$ in \eqref{eq:beta-}. 
We define $\widetilde{S}^\delta_t = \sup_{0\le s\le t} (\widetilde{B}_s + \delta s)$, $V_t:=\widetilde{S}^\delta_t-(\widetilde{B}_t + \delta t)$ and then
\begin{align}\label{eq:R- extension}
\widetilde{\cR}^{-,x}_t(x) = 
\begin{cases}
    \cR^{-,x}_t(x), &\quad 0\le t<A^{-,x}_{\infty},\\
    -V_{t-A^{-,x}_\infty},&\quad t\ge A^{-,x}_\infty.
\end{cases}    
\end{align}
Thanks to \eqref{eq:R- delta>0} and \cite[Theorem 1]{Graversen2000extension}, $\widetilde{\cR}^{-,x}(x)\le 0$ is an extension of $\cR^{-,x}(x)$ which satisfies the equation
\begin{align}\label{eq:tildeR-}
    -\widetilde{\cR}^{-,x}_t(x) = \frac1\mu x^- - \beta^{-,x}_t - \delta t +\left(1-\frac1\mu\right) I_t(-\widetilde{\cR}^{-,x}(x)) + \frac12 L_{-\widetilde{\cR}^{-,x}(x)}(t,0),\quad t\ge 0.
\end{align}
Note that $\cR^{-,x}_0(x)=-x^-$ and that $\cR_t(x) = x + X_t$ for $0\le t\le \cT_x$. Hence $\cT_x = T_0(-\widetilde\cR^{-,x}(x))$ for all $x\in\bR$.
By Skorokhod's lemma (see for example \cite[Lemma VI.2.1]{revuz2013continuous}), we deduce that for all $x\in\bR$, 
\begin{align*}
    -\widetilde\cR^{-,x}_t(x) &= -\cR^{-,x}_t(x) =  \frac{1}{\mu} x^- - \beta^{-,x}_t - \delta t + \left(1-\frac1\mu\right) I_t(- \cR^{-,x}(x)), &0\le t\le \cT_x,\\
    -\widetilde\cR^{-,x}_t(x) &=  -(\beta^{-,x}_t-\beta^{-,x}_{\cT_x}) - \delta (t-\cT_x) + \sup_{\cT_x\le s\le t} \big((\beta^{-,x}_s-\beta^{-,x}_{\cT_x}) + \delta\, (s-\cT_x) \big)^+, &t\ge \cT_x.
\end{align*}
In the recurrent case, \eqref{eq:R- delta>0} holds for all $t\ge 0$ and we write $\widetilde{\cR}^{-,x}_t(x):=\cR^{-,x}_t(x)$. Define $\rF^{-,x}_t := \sigma(\widetilde{\cR}^{-,x}_s(x):0\le s\le t)$ for all $t\ge 0$. Then it follows that 
\begin{align}
    \rF^{-,x}_t &= \sigma(\beta^{-,x}_s:0\le s\le t).\label{eq:filtration tildeR-}
\end{align}
We moreover define 
\begin{align}
    \rE^+_x := \rF^{+,x}_\infty=\sigma\big(\cR^{+,x}_t(x): t\ge 0\big), \qquad \rG^{-,x}_t := \rE^+_x \bigvee \rF^{-,x}_t,\label{eq:E+}\\
    \rE^-_x := \rF^{-,x}_\infty=\sigma\big(\widetilde{\cR}^{-,x}_t(x): t\ge 0\big), \qquad \rG^{+,x}_t := \rE^-_x \bigvee \rF^{+,x}_t. \label{eq:E-}
\end{align}
We can recover the process $\cR(x)$ by gluing the excursions away from 0 of $\cR^{+,x}(x)$ and $\tcR^{-,x}(x)$ ordered by their local times at 0; note that in the transient case, the construction will end when the local time reaches $L_\infty^x=L_{\cR^{+,x}(x)}(\infty,0)$. It implies that $\rG^{+,x}_t$ and $\rG^{-,x}_t$ are enlargements of $\rB_{\alpha^{+,x}_t}$ and $\rB_{\alpha^{-,x}_t}$ respectively.

In the case $x>0$, on the event $\{A^{-,x}_\infty=0\}$, $\widetilde{\cR}^{-,x}(x)$ depends only on 
$\widetilde{B}$ and is therefore independent of $\cR^{+,x}(x)$. Motivated by the method of enlargement of filtration, one would expect that the two processes $\widetilde{\cR}^{-,x}(x)$ and $\cR^{+,x}(x)$ are actually independent. This is indeed correct as shown in the following proposition.

\begin{proposition}\label{p:independence R+&R-}
Fix $x\in\bR$. The two processes $\cR^{+,x}(x)$ and $\widetilde{\cR}^{-,x}(x)$ are independent. Moreover, in the transient case,
\begin{align}\label{eq:description A-}
    A^{-,x}_{\infty} = \inf\{t>0: L_{-\widetilde{\cR}^{-,x}(x)}(t,0) > L_{\cR^{+,x}(x)}(\infty,0)\}.
\end{align}
In particular, $A^{-,x}_{\infty}$ is a $\rG^{-,x}$-stopping time.
\end{proposition}

\begin{proof}
We first prove the independence. It suffices to deal with the case $x\ge 0$, since otherwise we can apply the strong Markov property of $X$ at time $\cT_x=T_{-x}(X)$, which would imply that $(\cR^{-,x}_t(x):0\le t\le \cT_x)$ is independent of $(\cR_{t+\cT_x}(x):t\ge 0)$ hence also of $\cR^{+,x}(x)$ and $\big(\widetilde{\cR}^{-,x}_{t+\cT_x}(x): t\ge 0\big)$. We follow the method of \cite[Proposition 2.1]{perman1997perturbed} and approximate the processes $\cR^{+,x}(x)$ and $\widetilde \cR^{-,x}(x)$ using up- and down-crossings by $\cR(x)$ (and $V$ in the transient case). 
We will only deal with the transient case since the argument in the recurrent case is similar and indeed simpler.

For all $\varepsilon > 0$, we define the $\rB$-stopping times $T^\varepsilon_n$ and $S^\varepsilon_n$ as follows: $S^\varepsilon_0=T^\varepsilon_1 = 0$ and for all $n \ge 1$,
\begin{align}\label{eq:S and T eps}
    S^\varepsilon_n = \inf\{ t > T^\varepsilon_n: \cR_t(x) = -\varepsilon \},\qquad T^\varepsilon_{n+1} = \inf\{ t > S^\varepsilon_n: \cR_t(x) = 0 \}.
\end{align}
 Then there exists some (random) $N\in\bN$ such that $S^\varepsilon_N<\infty$ while $S^\varepsilon_{N+1}=\infty$. 
We also define the clocks corresponding to the up- and down-crossings from level $-\varepsilon$ to 0 by $\cR(x)$:
\begin{align}\label{eq:A eps}
    A^{\varepsilon+}_t = \sum_{n \ge 1} \big(\min(t, S^\varepsilon_n) - \min(t, T^\varepsilon_n) \big), \qquad
    A^{\varepsilon-}_t = \sum_{n \ge 1} \big( \min(t, T^\varepsilon_{n+1}) - \min(t, S^\varepsilon_n) \big). 
\end{align}
Let $\alpha^{\varepsilon+}$ and $\alpha^{\varepsilon-}$ denote their respective right-continuous inverses, and set
\begin{align*}
    \cR^{\varepsilon+}_u &= \cR_{\alpha^{\varepsilon+}_u}(x), \qquad u\ge 0,\\
    \cR^{\varepsilon-}_u &= \cR_{\alpha^{\varepsilon-}_u}(x), \qquad 0\le u < A^{\varepsilon-}_\infty.
\end{align*}
We define $\widetilde{T}^\varepsilon_n$, $\widetilde{S}^\varepsilon_n$ and $\widetilde A^{\varepsilon-}$ as in \eqref{eq:S and T eps} and \eqref{eq:A eps} but with $-V$ instead of $\cR(x)$, then $\widetilde{\alpha}^{\varepsilon-}$ the right-continuous inverse of $\widetilde{A}^{\varepsilon-}$ and $V^{\varepsilon-}_u=V_{\widetilde{\alpha}^{\varepsilon-}_u}$ for all $u\ge 0$. We extend $\cR^{\varepsilon-}_u$ to all $u\ge 0$ as follows: 
\begin{align*}
    \widetilde{\cR}^{\varepsilon-}_u = \begin{cases}
        \cR^{\varepsilon-}_u , & 0\le u<A^{\varepsilon-}_\infty,\\
        -V^{\varepsilon-}_{u-A^{\varepsilon-}_\infty}, &  u\ge A^{\varepsilon-}_\infty.
    \end{cases}
\end{align*}
We observe that 
\begin{enumerate}[(i)]

\item $\widetilde{\cR}^{\varepsilon-}$ can only jump from 0 to $-\varepsilon$. 
Let $J_1=0$ and $J_{k+1}$ be the $k$-th jump time of $\widetilde{\cR}^{\varepsilon-}$. Then $\widetilde{\cR}^{\varepsilon-}$ is measurable with respect to the sequence $(\widetilde{\cR}^{\varepsilon-}_{u+J_k}: 0\le u< J_{k+1}-J_k)_{k\ge 1}$.

\item For each $k\le N$, we have $J_{k+1}-J_k = T^{\varepsilon}_{k+1}-S^\varepsilon_k$ by the definition of $S^\varepsilon_k$ and $T^\varepsilon_{k+1}$ in \eqref{eq:S and T eps}. From \eqref{eq:bbflow},
\[
\widetilde{\cR}^{\varepsilon-}_{t+J_k} = \cR_{t+S^\varepsilon_k}(x) = -\varepsilon + X_{t+S^\varepsilon_k} - X_{S^\varepsilon_k},\qquad \forall\, 0\le t<J_{k+1}-J_k.
\]
For $k>N$, similarly we have $\widetilde{\cR}^{\varepsilon-}_{t+J_k} = -V_{t+\widetilde S^\varepsilon_{k-N}}$ for all $0\le t<J_{k+1}-J_k=\widetilde T^{\varepsilon}_{k-N+1}-\widetilde S^\varepsilon_{k-N}$.

\item On the event $\{N\ge k\}=\{S_k^\varepsilon<\infty\}$, the process $(\cR_{t+T^\varepsilon_k}(x)-\cR_{T^\varepsilon_k}(x):0\le t<S^\varepsilon_k-T^\varepsilon_k)$ is measurable with respect to the process $(X_{t+T^\varepsilon_k}-X_{T^\varepsilon_k}:0\le t<S^\varepsilon_k-T^\varepsilon_k)$ by the adaptedness of the solution to \eqref{eq:bbflow} (with $X$ replaced by $X_{T^\varepsilon_k+\cdot}-X_{T^\varepsilon_k}$).
In addition, for each $n\in\bN$,
\[
\{N=n\} = \{S^\varepsilon_n<\infty\}\cap\Big\{\inf_{t\ge 0}\, \cR_{t+T^\varepsilon_{n+1}}(x) > -\varepsilon\Big\}.
\]
By the adaptedness of the solution to \eqref{eq:bbflow} (with $X$ replaced by $X_{T^\varepsilon_{n+1}+\cdot}-X_{T^\varepsilon_{n+1}}$), the second event on the right-hand side is measurable with respect to $X_{T^\varepsilon_{n+1}+\cdot}-X_{T^\varepsilon_{n+1}}$ on the event $\{S^\varepsilon_n<\infty\}$.
Moreover, $X_{T^\varepsilon_{k+1}}-X_{S^\varepsilon_k}=\varepsilon$ on the event $\{N\ge k\}$ by (ii). 
We conclude that $\cR^{\varepsilon+}$ is measurable with respect to the sequence $(X_{t+T^\varepsilon_k}-X_{T^\varepsilon_k}:0\le t<S^\varepsilon_k-T^\varepsilon_k)_{1\le k\le N+1}$.

\item For all $k\le N$, $M(X)$ cannot increase in $[S^\varepsilon_k,T^\varepsilon_{k+1}]$ by Lemma~\ref{l:inclusion}. Therefore, on the event $\{N\ge k\}$, from \eqref{eq:X by BM} and the strong Markov property of $\big(X,M(X)\big)$ at time $S^\varepsilon_k$, the process $(X_{t+S^\varepsilon_k} - X_{S^\varepsilon_k}: 0\le t<T^{\varepsilon}_{k+1}-S^\varepsilon_k)$ is independent of $(X_t: 0\le t\le S^\varepsilon_k)$ and $(X_{t+T^\varepsilon_{k+1}}-X_{T^\varepsilon_{k+1}}: t\ge 0)$, and is distributed as a Brownian motion with drift $\delta$ starting at 0 and absorbed at $\varepsilon$.

\item The process $V\eqdis |U|$ from \eqref{eq:U-B}. Therefore, by \cite[Theorem 2]{Graversen2000extension}, the process $(-V_{t+\widetilde S^\varepsilon_k}: 0\le t<\widetilde T^{\varepsilon}_{k+1}-\widetilde S^\varepsilon_k)$ is distributed as a Brownian motion with drift $\delta$ starting at $-\varepsilon$ and absorbed at 0. 
\end{enumerate}

We now prove that $\cR^{\varepsilon+}$ and $\widetilde \cR^{\varepsilon-}$ are independent. By (i) and (iii), it suffices to prove the independence between $(X_{t+T^\varepsilon_k} - X_{T^\varepsilon_k}:0\le t<S^\varepsilon_k-T^\varepsilon_k)_{1\le k\le N+1}$ and $(\widetilde{\cR}^{\varepsilon-}_{u+J_k}: 0\le u< J_{k+1}-J_k)_{k\ge 1}$. 
Recall that $V$ is independent of $\cR(x)$ (because $\widetilde{B}$ is independent of $B$). It together with (ii), (iv) and (v) implies that $(\widetilde{\cR}^{\varepsilon-}_{u+J_k}: 0\le u< J_{k+1}-J_k)_{k\ge 1}$ is a sequence of i.i.d. processes. 
We deduce the desired independence by combining these facts with the independence stated in (iv) and the reformulation of the event $\{N=n\}$ in (iii).

\medskip

Finally, we have for all $t\ge 0$,
\[
|A^{\varepsilon+}_t-A^{+,x}_t| = |A^{\varepsilon-}_t-A^{-,x}_t|\le \int_0^t \ind_{\{\cR_s(x)\in[-\varepsilon,0]\}} \dd s \to 0, \qquad \varepsilon\downarrow 0,
\]
since the set of zeros of $\cR_t(x)$ almost surely has Lebesgue measure 0 by the occupation times formula \eqref{eq:occupation times formula}. Similarly, we have that almost surely, as $\varepsilon\downarrow 0$, $\widetilde{A}^{\varepsilon-}_t-t$ converges to 0 uniformly on every compact interval. 
Since $\varepsilon\mapsto\widetilde{A}^{\varepsilon-}_t$ is decreasing, it follows that almost surely, $\widetilde{\alpha}^{\varepsilon-}_u\to u$ uniformly on every compact interval. Hence for all fixed $u>0$, on the event $\{A^{-,x}_\infty\le u\}$, we have $A^{\varepsilon-}_\infty\le u$ and
\[
\widetilde{\cR}^{\varepsilon-}_u = -V^{\varepsilon-}_{u-A^{\varepsilon-}_\infty} = -V_{\widetilde\alpha^{\varepsilon-}_{u-A^{\varepsilon-}_\infty}} \to -V_{u-A^{-,x}_\infty} \quad \text{as}\quad \varepsilon\to 0.
\]
On the event $\{A^{-,x}_\infty>u\}$, we have $A^{\varepsilon-}_\infty>u$ for $\varepsilon$ sufficiently small. We observe that for all $t\ge 0$, $A^{\varepsilon-}_t\le A^{-,x}_t$ and the map $\varepsilon\mapsto A^{\varepsilon-}_t$ is decreasing. Therefore $\alpha^{\varepsilon-}_u\ge \alpha^{-,x}_u$, and $\varepsilon\mapsto\alpha^{\varepsilon-}_u$ is increasing. It follows that $\alpha^{\varepsilon-}_u$ converges to some $\alpha^{-,*}\ge \alpha^{-,x}_u$ as $\varepsilon\downarrow 0$. If $\alpha^{-,*}>\alpha^{-,x}_u$, then $A^{-,x}_{\alpha^{-,*}}>u$ hence $A^{\varepsilon-}_{\alpha^{\varepsilon-}_u}>u$ for sufficiently small $\varepsilon>0$, yielding a contradiction. We deduce that almost surely, $\alpha^{\varepsilon-}_u\downarrow \alpha^{-,x}_u$ as $\varepsilon\downarrow 0$. We conclude that $\tcR^{\varepsilon-}_u \to \tcR^{-,x}_u(x)$ as $\varepsilon\downarrow 0$.

On the other hand, we observe that for all $u\ge 0$, $\alpha^{\varepsilon+}_u\le \alpha^{+,x}_u$ and $\varepsilon\mapsto\alpha^{\varepsilon+}_u$ is decreasing. Therefore $\alpha^{\varepsilon+}_u\uparrow \alpha^{+,*}\le \alpha^{+,x}_u$ as $\varepsilon\downarrow 0$.
If $A^{+,x}_{\alpha^{+,*}}< u$, then $A^{\varepsilon+}_{\alpha^{\varepsilon+}_u}<u$ for $\varepsilon>0$ sufficiently small, a contradiction. Therefore $\alpha^{+,*}\ge \alpha^{+,x}_{u-}$, implying that $\alpha^{+,*}= \alpha^{+,x}_u$ for $u$ a continuity point of $\alpha^{+,x}$. 
If $\alpha^{+,x}_{u-}< \alpha^{+,x}_u$, then necessarily $\cR_{\alpha^{+,x}_{u-}}(x)=\cR_{\alpha^{+,x}_u}(x)=0$. We now prove that $\cR_{\alpha^{+,*}}(x)=0$ in this case. 
Suppose $\cR_{\alpha^{+,*}}(x)\ne 0$. Then $\alpha^{+,*}\in(\alpha^{+,x}_{u-}, \alpha^{+,x}_u)$ hence $\cR_{\alpha^{+,*}}(x)< 0$. However, we have $\cR_{\alpha^{\varepsilon+}_u}(x)\ge -\varepsilon$ thus $\cR_{\alpha^{+,*}}(x)=\lim_{\varepsilon\downarrow 0} \cR_{\alpha^{\varepsilon+}_u}(x)\ge 0$, again a contradiction. We conclude that almost surely, $\cR^{\varepsilon+}_u=\cR_{\alpha^{\varepsilon+}_u}(x)\to \cR_{\alpha^{+,x}_u}(x) = \cR^{+,x}_u(x)$ for all $u\ge 0$.
In particular, the finite-dimensional distributions of $(\cR^{\varepsilon+},\tcR^{\varepsilon-})$ converge to those of $(\cR^{+,x}(x),\tcR^{-,x}(x))$ as $\varepsilon\to 0$.
It implies the independence between $\cR^{+,x}(x)$ and $\widetilde{\cR}^{-,x}(x)$.

To finish the proof, we observe that 
$$
L_{-\widetilde\cR^{-,x}(x)}({A^{-,x}_\infty},0)=L^x_\infty=L_{\cR^{+,x}(x)}(\infty,0).
$$ 
Hence \eqref{eq:description A-} follows since $L_{-\widetilde\cR^{-,x}(x)}(\cdot,0)$ will immediately increase after $A^{-,x}_\infty$ by the construction of $\widetilde{\cR}^{-,x}(x)$.
\end{proof}

Recall that $\rho(t,x)=\kappa(L_t^x)$ as in Proposition~\ref{p:biconti L}~(ii). We now set 
\begin{align*}
    \rho^{+,x}(t,y)&:=\rho({\alpha^{+,x}_t},y\big), \quad Y^{+,x}_t=Y_{\alpha_t^{+,x}}, \qquad\quad t\ge 0,\\
    \rho^{-,x}(t,y)&:=\rho({\alpha^{-,x}_t},y\big), \quad Y^{-,x}_t=Y_{\alpha_t^{-,x}}, \qquad\quad 0\le t<A^{-,x}_\infty.
\end{align*}
Observe that $Y^{+,x}$ and $Y^{-,x}$ are continuous processes, and that $\rho^{+,x}(t,y)$ (resp. $\rho^{-,x}(t,y)$) is the local time at $y\le x$ (resp. $y\ge x$) of $Y^{+,x}_t$ (resp. $Y^{-,x}_t$). 
We can rewrite the definition of $A_t^{+,x}$ as the occupation time of the process $(Y_t)$ \emph{below} the level $x$ up to time $t$. Then intuitively speaking, $Y^{+,x}$ is constructed by gluing the excursions of $Y$ below $x$. Similarly, $Y^{-,x}$ is constructed by gluing the excursions of $Y$ above $x$.

\begin{proposition}\label{p:filtration}
 We have
\begin{align}\label{eq:R+&Y+}
    \rF^{+,x}_t = \sigma(Y^{+,x}_s : 0\le s\le t).
\end{align}
In particular, $(\rE^+_x: x\in \bR)$ defined in \eqref{eq:E+} is a filtration. In the recurrent case, $(\rE^-_{-x}: x\in \bR)$ is also a filtration.
\end{proposition}

\begin{proof}
We refer to the proof of \cite[(3.9)]{hu-warren}. First we obtain from \eqref{eq:derivative}--\eqref{eq:occupied} that
\[
\cR^{+,x}_t(x) = \cR_{\alpha^{+,x}_t}(x) - \cR_{\alpha^{+,x}_t}(Y_{\alpha^{+,x}_t}) =  \int_{Y_{\alpha^{+,x}_t}}^x \big(1+\lambda \rho(\alpha^{+,x}_t,y)\big) \dd y = \int_{Y^{+,x}_t}^x \big(1+\lambda\rho^{+,x}(t,y)\big) \dd y.
\]
It implies that 
\[
\sigma(\cR^{+,x}_s(x):0\le s\le t) \subseteq \sigma(Y_{\alpha_s^{+,x}} : 0\le s\le t).
\]
Meanwhile, using \eqref{eq:bbflow} with $t$ replaced by $\alpha^{+,x}_t$, we obtain that for every $y\le x$,
\begin{align}\label{eq:R y}
\cR^{+,x}_t(y) = y-x + \cR^{+,x}_t(x) - \lambda\int_0^t \ind_{\{\cR^{+,x}_s(y)\le 0\}} \dd s.
\end{align}
By \cite[Proposition V.1.4]{revuz2013continuous}, $\cR^{+,x}(x)=\cR_{\alpha^{+,x}}(x)$ is a $\rB_{\alpha^{+,x}}$-semimartingale. Therefore pathwise uniqueness holds for \eqref{eq:R y} by \cite[Theorem V.6]{protter2005stochastic} and a removal of drift (as applied in Proposition~\ref{p:bbflow}), and thus $\cR^{+,x}(y)$ is $\rF^{+,x}$-adapted. We deduce that 
\[
\{Y^{+,x}_t<y\} = \{\cR^{+,x}_t(y)>0\} \in \sigma(\cR^{+,x}_s(x):0\le s\le t).
\]
We deduce the other direction and therefore \eqref{eq:R+&Y+}. 
The result for $(\rE^-_{-x}: x\in \bR)$ in the recurrent case follows similarly.
\end{proof}

\begin{remark}\label{rmk:Y-}
The proof of Proposition~\ref{p:filtration} provides an extension of $Y^{-,x}$ in the transient case as follows. We consider the equation
\begin{align}\label{eq:tcR y}
\tcR^{-,x}_t(y) = y- x + \tcR^{-,x}_{t}(x) + \lambda\int_0^t \ind_{\{\tcR^{-,x}_s(y)> 0\}} \dd s,\qquad t\ge 0.
\end{align}
From \eqref{eq:tildeR-}, the process $\tcR^{-,x}(x)$ behaves as a perturbed drifted Brownian motion of drift $\delta$ starting at $-x^-$ until hitting 0, then as the opposite of a $\mathrm{RBM}(-\delta)$ process. Therefore \eqref{eq:tcR y} admits a unique strong solution, and \eqref{eq:tcR y} induces a flow of homeomorphisms of $\bR$. Let $\widetilde{Y}^{-,x}_t$ be such that $\tcR^{-,x}_t(\widetilde{Y}^{-,x}_t)=0$. Since $\tcR^{-,x}(x)\le 0$, we have $\widetilde{Y}^{-,x}\ge x$. Since $\tcR^{-,x}_t(x)=\cR^{-,x}_t(x)$ for $t<A^{-,x}_\infty$, we derive from \eqref{eq:bbflow} that for all $t<A^{-,x}_\infty$ and $y\ge x$, $\tcR^{-,x}_t(y)=\cR^{-,x}_t(y)$ hence $\widetilde Y^{-,x}_t = Y^{-,x}_t$. Finally, by Proposition~\ref{p:independence R+&R-}, $\widetilde Y^{-,x}$ is independent of $\rE_x^+$.
\end{remark}

We close this section with two useful time-change formulas, analogous to \cite[Corollary 2.5]{aidekon2024infinite}.
\begin{proposition}
For any random function $g(\ell,y)$ satisfying that $g(\rho(t,Y_t),Y_t)$ is $\rB$-progressive and $\bE[\int_{\bR_+\times\bR}g^2(\ell,x)\dd\ell\dd x]<\infty$, we have:
\begin{enumerate}[(i)]
    \item The process $t\mapsto g(\rho^{+,x}(t,Y^{+,x}_t),Y^{+,x}_t)$ is $\rG^{+,x}$-progressive, and
    \begin{align}
    \int_0^{\infty}g(\rho(t,Y_t),Y_t)\ind_{\{\cR_t(x)>0\}}\dd \beta_t &= \int_0^{\infty}g(\rho^{+,x}(t,Y^{+,x}_t),Y^{+,x}_t)\dd \beta^{+,x}_t.\label{eq:posi time change}
    \end{align}
    
    \item The process $t\mapsto g(\rho^{-,x}(t,Y^{-,x}_t),Y^{-,x}_t)\ind_{\{A^{-,x}_\infty>t\}}$ is $\rG^{-,x}$-progressive, and
\begin{align}
    \int_0^{\infty} g(\rho(t,Y_t),Y_t) \ind_{\{\cR_t(x)\leq 0\}} \dd \beta_t &= \int_0^\infty g(\rho^{-,x}(t,Y^{-,x}_t),Y^{-,x}_t)\ind_{\{A^{-,x}_\infty>t\}} \dd \beta^{-,x}_t.\label{eq:nega time change}
\end{align}
\end{enumerate}
\end{proposition}

Note that $\ind_{\{A^{-,x}_\infty>t\}}\equiv 1$ in the recurrent case.

\begin{proof}
By \cite[Proposition V.1.4]{revuz2013continuous}, the process $(g(\rho^{+,x}(t,Y^{+,x}_t), Y^{+,x}_t):t\ge 0)$ is $\rB_{\alpha^{+,x}}$-progressive, hence $\rG^{+,x}$-progressive since the latter filtration is larger. By the time-change formula for the stochastic integral \cite[Proposition V.1.5]{revuz2013continuous}, 
    \[
    \int_0^{\alpha^{+,x}_t}g(\rho(s,Y_s),Y_s)\ind_{\{\cR_s(x)>0\}}\dd \beta_s = \int_0^t g(\rho^{+,x}(s,Y^{+,x}_s),Y^{+,x}_s)\dd \beta^{+,x}_s.
    \]
Letting $t\to \infty$ yields \eqref{eq:posi time change}. Equation~\eqref{eq:nega time change} follows similarly in the recurrent case. In the transient case, we notice that by definition of $\beta^{-,x}$,
    \[
    \int_0^{\alpha^{-,x}_t} \ind_{\{\cR_s(x)\leq 0\}} \dd \beta_s = \beta^{-,x}_{t\wedge A^{-,x}_\infty}.
    \]
    Recall that $\beta^{-,x}$ is a Brownian motion independent of $\rE^+_x$ by Proposition~\ref{p:independence R+&R-}, and that $A^{-,x}_\infty$ is a $\rG^{-,x}$-stopping time by Proposition~\ref{p:independence R+&R-}. Therefore $(\beta^{-,x}_{t\wedge A^{-,x}_\infty}:t\ge 0)$ is a $\rG^{-,x}$-martingale. Recall the definition of $Y^{-,x}$. Then we have
    \[
    g(\rho^{-,x}(t,Y^{-,x}_t),Y^{-,x}_t)\ind_{\{A^{-,x}_\infty>t\}} = g(\rho(\alpha^{-,x}_t,Y_{\alpha^{-,x}_t}),Y_{\alpha^{-,x}_t})\ind_{\{A^{-,x}_\infty>t\}}.
    \]
    It implies that $t\mapsto g(\rho^{-,x}(t,Y^{-,x}_t),Y^{-,x}_t)\ind_{\{A^{-,x}_\infty>t\}}$ is $\rG^{-,x}$-progressive again by \cite[Proposition V.1.4]{revuz2013continuous}. Equation~\eqref{eq:nega time change} then follows by applying a generalization of \cite[Proposition V.1.5]{revuz2013continuous} in our setting. Indeed, by \cite[Proposition IV.1.13]{revuz2013continuous}, we deduce that for all $t,u\ge 0$,
    \[
    \int_0^t \ind_{\{\alpha_s^{-,x}\le u\}} \dd \beta^{-,x}_s = \beta^{-,x}_{t\wedge A^{-,x}_u} = \int_0^{\alpha^{-,x}_t\wedge u} \ind_{\{\cR_s(x)\leq 0\}} \dd \beta_s.
    \]
    Equation \eqref{eq:nega time change} therefore follows by approximating $t\mapsto g(\rho(t,Y_t),Y_t)$ by simple processes. 
\end{proof}

\section{White Noise and Stochastic Integral}\label{s:integral}
In this section, we give the explicit description of the white noise $\cW$ appearing in Theorems~\ref{thm:rk} and \ref{thm:rk recur}. Recall that $Y_t=\cR^{-1}_t(0)$. We first deal with the transient case.

\begin{lemma}\label{l:WN}
Let $\widetilde{\cW}$ be a white noise in $\bR_+\times\bR$ independent of $B$. For any $g\in L^2(\bR_+\times\bR)$ with compact support, define $\widetilde{g}(\ell,x)=g(\ell,x)\ind_{[\rho(\infty,x),\infty)}(\ell)$ and
\begin{equation}\label{eq:WN}
    \cW(g) := \int_0^{\infty}g(\rho(t,Y_t),Y_t)\dd \beta_t + \int_{\bR_+\times\bR} \widetilde{g}(\ell,x)\widetilde{\cW}(\dd\ell,\dd x),
\end{equation}
where the second part in \eqref{eq:WN} is the stochastic integral with respect to $\widetilde{\cW}$ as in Section \ref{s:WN}. Then $\cW$ is also a white noise on $\bR_+\times\bR$.
\end{lemma}

\begin{proof}
Recall that $\mathscr{B}_\infty$ is the $\sigma$-algebra generated by the Brownian motion $\beta$. Then conditionally on $\mathscr{B}_\infty$, the quadratic variation  of the martingale $\big(\int_{\bR_+\times(-\infty,y]}\widetilde{g}(\ell,x)\widetilde{\cW}(\dd\ell,\dd x):y\in\bR)$ is 
\[
\bigg(\int_{-\infty}^y\dd x\int_{\rho(\infty,x)}^{\infty}g^2(\ell,x)\dd \ell: y\in\bR \bigg),
\]
which is a bounded nonnegative process. On the other hand, by a generalization of the occupation times formula \eqref{eq:occupied} as shown in \cite[Exercise VI.1.15]{revuz2013continuous}, the quadratic variation of  $\big(\int_0^T g(\rho(t,Y_t),Y_t)\dd \beta_t: T\ge 0\big)$ is 
\[
\bigg(\int_{\bR}\dd x\int_0^{\rho(T,x)}g^2(\ell,x)\dd \ell: T\ge 0\bigg).
\]
Hence by Novikov's criterion (see for example \cite[Corollary 3.5.13]{karatzas1991brownian}), we obtain that for every $\alpha\in\bR$,
\begin{align*}
     &\quad\ \bE\Big[e^{\alpha \cW(g)-\frac{\alpha^2}{2}\Vert g\Vert _{L^2}^2} \Big]\\
     &= \bE\left[e^{\alpha \int_0^{\infty}g(\rho(t,Y_t),Y_t)\dd \beta_t-\frac{\alpha^2}{2} \int_{\bR}\dd x\int_0^{\rho(\infty,x)}g^2(\ell,x)\dd \ell}
     \bE\Big[ e^{\alpha \int_{\bR_+\times\bR} \widetilde{g}(\ell,x)\widetilde{\cW}(\dd\ell,\dd x)-\frac{\alpha^2}{2}\int_{\bR}\dd x\int_{\rho(\infty,x)}^{\infty}g^2(\ell,x)\dd \ell} \Big|\mathscr{B}_\infty\Big] \right]\\
     &=\bE\Big[e^{\alpha \int_0^{\infty}g(\rho(t,Y_t),Y_t)\dd \beta_t-\frac{\alpha^2}{2}\int_0^{\infty}g^2(\rho(t,Y_t),Y_t)\dd t}\Big]=1,
\end{align*}
implying that $\cW(g)\sim \mathcal{N}(0,||g||_{L^2}^2)$. Therefore $\cW$ is a white noise.
\end{proof}

Recall the definition \eqref{eq:mart_WN} of the martingale measure $\cM$ associated with the white noise $\cW$ on $\bR_+\times\bR$.

\begin{proposition}\label{p:mart measure}
For every $r\in\bR$, $(\cM_x: x\ge r)$ is a continuous martingale measure with respect to the filtration $(\widetilde{\rE}^+_x: x\ge r)$, where
\[
\widetilde{\rE}^+_x:=\rE^+_x\bigvee \sigma\big(\widetilde{\cW}(\ind_{A\times[r,u]}): A\in\mathcal{L},\, r\le u\le x\big).
\]
\end{proposition}

\begin{proof}
Since $\cW$ is a white noise, it suffices to show that for every $A\in\cL$, the process $(\cM_x(A): x\ge r)$ is a continuous martingale with respect to the filtration $(\widetilde{\rE}^+_x: x\ge r)$.
From \eqref{eq:posi time change}, we can rewrite $\cM_x(A)$ as 
\begin{equation*}
 \cM_x(A) = \int_0^{\infty} \ind_{A\times[r,x]}(\rho^{+,x} (t,Y^{+,x}_t),Y^{+,x}_t)\dd \beta^{+,x}_t + \int_r^x \widetilde{\cW}(A\cap[\rho({\infty},y),\infty),\dd y).
\end{equation*}
Recall that $\beta^{+,x}$ is an $\rF^{+,x}$-adapted Brownian motion. It together with the above equality and \eqref{eq:R+&Y+} implies the $\widetilde{\mathscr{E}}^+_x$-measurability of $\cM_x(A)$. From \eqref{eq:nega time change}, for every $z>x$ and $A\in\cL$ we can express $\cM_z(A)-\cM_x(A)$ as 
\begin{equation}\label{eq:M difference}
\begin{aligned}
&\cM_z(A)-\cM_x(A) \\
= &\int_0^\infty\ind_{A\times(x,z]}(\rho^{-,x}(t,Y^{-,x}_t),Y^{-,x}_t) \ind_{\{A_{\infty}^{-,x}>t\}}\dd \beta^{-,x}_t + \int_x^z \widetilde{\cW}(A\cap[\rho({\infty},y),\infty),\dd y).
\end{aligned}
\end{equation}
Conditionally on $\rB_\infty$, the second part of \eqref{eq:M difference} is a centred Gaussian random variable independent of the trace of $\widetilde{\cW}$ on $\bR_+\times[r,x]$. On the other hand, for any $F\in\rE_x^+$, the process 
\[
\left(\int_0^t \ind_{F\cap \{A_{\infty}^{-,x}>s\}} \ind_{A\times(x,z]}(\rho^{-,x}(s,Y^{-,x}_s),Y^{-,x}_s)\dd \beta^{-,x}_s: t\ge 0 \right)
\]
is a $\rG^{-,x}$-martingale by Proposition~\ref{p:independence R+&R-} and Remark~\ref{rmk:Y-}, which is indeed $L^2$-bounded by the occupation times formula applied to the process $Y^{-,x}$, hence uniformly integrable. Since the white noise $\widetilde{\cW}$ is independent of $B$, we deduce that
\begin{align*}
    &\bE\left[\int_0^\infty\ind_{A\times(x,z]}(\rho^{-,x}(t,Y^{-,x}_t),Y^{-,x}_t)\ind_{\{A_{\infty}^{-,x}>t\}}\dd \beta^{-,x}_t\,\Big|\, \widetilde{\rE}^+_x\right] =0.
\end{align*}
For the same reason, it follows that
\[\bE\left[\int_x^z \widetilde{\cW}(A\cap[\rho({\infty},y),\infty),\dd y) \,\Big|\, \widetilde{\rE}^+_x\right] = \bE\left[\int_x^z \widetilde{\cW}(A\cap[\rho({\infty},y),\infty),\dd y) \,\Big|\, \sigma(\widetilde{\rE}^+_x,\rB_\infty)\right] =0.\]
We deduce that $\mathbb{E}[\cM_z(A)-\cM_x(A)\,|\, \widetilde{\mathscr{E}}^+_x]=0$ for $x<z$, which completes the proof.
\end{proof}

We expect that for $\rE_x^+$-predictable $f\in\rL^2$, the integral $f\cdot \cM$ shares the same structure as \eqref{eq:WN}. As shown in the next proposition, this indeed holds whenever  $f$ admits a suitable version. The next proposition is an analog of \cite[Proposition 3.2]{aidekon2024infinite}.

\begin{proposition}\label{p:representation}
Suppose $f\in\rL^2$ is $(\mathscr{E}^+_x, x\ge r)$-predictable such that the process $(f(\rho(t,Y_t),Y_t): t\ge 0)$ admits a $\rB$-progressive version. Define $\widetilde{f}(\ell,x)=f(\ell,x)\ind_{[\rho(\infty,x),\infty)}(\ell)$. Then almost surely,
\begin{equation}\label{eq:represenation}
    f\cdot \cM=\int_0^{\infty}f(\rho(t,Y_t),Y_t)\ind_{[r,\infty)}(Y_t)\dd \beta_t + \int_{\bR_+\times[r,\infty)} \widetilde{f}(\ell,x) \widetilde{\cW}(\dd\ell,\dd x).
\end{equation}
\end{proposition}

\begin{proof}
Since $f\in\rL^2$, there exists a sequence of simple functions $\{f_n\}_{n\geq 1}$ such that $\Vert f_n-f\Vert_{\rL^2}\rightarrow 0$ as $n\rightarrow \infty$. By isometry, $f_n\cdot \cM-f \cdot \cM\rightarrow 0$ in $L^2$. Write $f_n$ of the form
\begin{equation}\label{eq:f_n}
    f_n(\ell,x)=\sum_{k,j=1}^{\infty} Z_{k,j}^n \ind_{A_j^n}(\ell) \ind_{(a_{k-1}^n,a_k^n]}(x),
\end{equation}
where for each $n$, $r=a_0^n<a_1^n<...<a_{k-1}^n<a_k^n<...$ is an increasing sequence that goes to infinity, $\{A_j^n\}_{j\geq 1}\subseteq \mathcal{L}$ is pairwise disjoint, and $Z_{k,j}^n$ is a bounded $\mathscr{E}^+_{a_{k-1}^n}$-measurable random variable for each $j, k\geq 1$ which will be zero for all $j, k$ large enough, meaning we sum over a finite index set in \eqref{eq:f_n}.

By the definition of the white noise $\cW$ as in \eqref{eq:WN} and linearity for simple functions, we can write
\begin{equation*}
    f_n\cdot \cM=\sum_{k,j=1}^{\infty} Z_{k,j}^n \left(\int_0^{\infty} \ind_{A_j^n}(\rho(s,Y_s)) \ind_{(a_{k-1}^n,a_k^n]}(Y_s) \dd \beta_s + 
    \int_{a_{k-1}^n}^{a_k^n}\widetilde{\cW}(A_j^n\cap[\rho(\infty,x),\infty),\dd x)\right).
\end{equation*}
From \eqref{eq:posi time change}--\eqref{eq:nega time change}, we have
\begin{align}
    &\int_0^{\infty} \ind_{A_j^n}(\rho(t,Y_t)) \ind_{(a_{k-1}^n,a_k^n]}(Y_t) \dd \beta_t = \int_0^{\infty}\ind_{A_j^n}(\rho^{+,v}(t,Y^{+,v}_t))\ind_{(a_{k-1}^n,a_k^n]}(Y^{+,v}_t)\dd \beta^{+,v}_t\label{posi-simple}\\
    =\, & \int_0^\infty\ind_{A_j^n}(\rho^{-,u}(t,Y^{-,u}_t))\ind_{(a_{k-1}^n,a_k^n]}(Y^{-,u}_t)\ind_{\{A^{-,u}_\infty>t\}}\dd \beta^{-,u}_t, \label{nega-simple}
\end{align}
with $u=a_{k-1}^n$ and $v=a_k^n$. Let $I(f)$ denote the right-hand side of \eqref{eq:represenation}. Then
\begin{align*}
    I(f)= \sum_{k=1}^{\infty}\left[\int_0^{\infty}f(\rho(t,Y_t),Y_t)\ind_{(a_{k-1}^n,a_k^n]}(Y_t)\dd \beta_t +  \int_{\bR_+\times(a_{k-1}^n,a_k^n]} \widetilde{f}(\ell,x) \widetilde{\cW}(\dd\ell,\dd x)\right].
\end{align*}
Again from \eqref{eq:posi time change}--\eqref{eq:nega time change}, we have
\begin{align}
    &\int_0^{\infty} f(\rho(t,Y_t),Y_t)\ind_{(a_{k-1}^n,a_k^n]}(Y_t)\dd \beta_t
    = \int_0^{\infty} f(\rho^{+,v}(t,Y^{+,v}_t),Y^{+,v}_t)\ind_{(a_{k-1}^n,a_k^n]}(Y^{+,v}_t)\dd \beta^{+,v}_t\label{posi-f}\\
    =\, & \int_0^\infty f(\rho^{-,u}(t,Y^{-,u}_t),Y^{-,u}_t)\ind_{(a_{k-1}^n,a_k^n]}(Y^{-,u}_t)\ind_{\{A^{-,u}_\infty>t\}}\dd \beta^{-,u}_t.\label{nega-f}
\end{align}
Write $h_n=f_n-f$. Observe that $Y^{-,u}$ is $\rG^{-,u}$-progressive from Remark~\ref{rmk:Y-}. Therefore the process 
\[
\left(Z_{k,j}^n\ind_{A_j^n}(\rho^{-,u}(t,Y^{-,u}_t))\ind_{(a_{k-1}^n,a_k^n]}(Y^{-,u}_t)\ind_{\{A^{-,u}_\infty>t\}}: t\ge 0\right)
\]
is $\rG^{-,u}$-progressive because $A^{-,u}_\infty$ is a $ \rG^{-,u}$-stopping time by Proposition~\ref{p:independence R+&R-} and $Z_{k,j}^n$ is $\rE^+_u$-measurable by definition. 
Applying \eqref{posi-simple}--\eqref{nega-f}, we derive that $f_n\cdot \cM-I(f)=\sum_{k=1}^{\infty}\Delta_n(k)$, where $\Delta_n(k)=({\rm I})+({\rm II})$ and 
\begin{align}
    ({\rm I})&= \sum_{j=1}^\infty Z_{k,j}^n\int_0^{\infty}\ind_{A_j^n}(\rho^{+,v}(t,Y^{+,v}_t))\ind_{(a_{k-1}^n,a_k^n]}(Y^{+,v}_t)\dd \beta^{+,v}_t\notag\\
    &\quad -\int_0^{\infty} f(\rho^{+,v}(t,Y^{+,v}_t),Y^{+,v}_t)\ind_{(a_{k-1}^n,a_k^n]}(Y^{+,v}_t)\dd \beta^{+,v}_t\label{posi-differ}\\
    &=\, \sum_{j=1}^\infty Z_{k,j}^n \int_0^\infty \ind_{A_j^n}(\rho^{-,u}(t,Y^{-,u}_t))\ind_{(a_{k-1}^n,a_k^n]}(Y^{-,u}_t)\ind_{\{A^{-,u}_\infty>t\}} \dd \beta^{-,u}_t \notag\\
    &\quad -\int_0^\infty f(\rho^{-,u}(t,Y^{-,u}_t),Y^{-,u}_t)\ind_{(a_{k-1}^n,a_k^n]}(Y^{-,u}_t) \ind_{\{A^{-,u}_\infty>t\}}\dd \beta^{-,u}_t \notag\\
    &=\, \int_0^\infty h_n(\rho^{-,u}(t,Y^{-,u}_t),Y^{-,u}_t)\ind_{(a_{k-1}^n,a_k^n]}(Y^{-,u}_t)\ind_{\{A^{-,u}_\infty>t\}} \dd \beta^{-,u}_t,  \label{nega-differ}\\ 
    ({\rm II})&= \sum_{j=1}^\infty Z_{k,j}^n \int_{a_{k-1}^n}^{a_k^n}\widetilde{\cW}(A_j^n\cap[\rho(\infty,x),\infty),\dd x) - \int_{\bR_+\times[a_{k-1}^n,a_k^n]} \widetilde{f}(\ell,x) \widetilde{\cW}(\dd\ell,\dd x)\\
    &= \int_{\bR_+\times[a_{k-1}^n,a_k^n]} \widetilde{h}_n(\ell,x)\widetilde{\cW} (\dd\ell,\dd x),\label{W-part}
\end{align}
with $u=a_{k-1}^n$ and $v=a_k^n$ in \eqref{posi-differ}-- \eqref{nega-differ} and $\widetilde h_n(\ell,x):= h_n(\ell,x)\ind_{[\rho(\infty,x),\infty)}(\ell)$ in \eqref{W-part}. Let us prove the $\widetilde{\mathscr{E}}^+_{a_k^n}$-measurability of $\Delta_n(k)$. Recall that $Z^n_{k,j}$ is $\rE^+_{a^n_{k-1}}$-measurable from the definition of $f_n$. By Proposition~\ref{p:filtration}, we obtain that $\rho(\infty,x)=\rho^{+,x}(\infty,x)$ is $\rE^+_x$-measurable and that $Y^{+,v}$ is $\rE^+_{a^n_k}$-measurable. Since $Y^{+,v}\le v= a_k^n$ and $f(\cdot,x)$ is $\rE^+_x$-measurable, the process $t\mapsto f(\rho^{+,v}(t,Y^{+,v}_t),Y^{+,v}_t)$ is also $\rE^+_{a^n_k}$-measurable. The desired measurability follows from \eqref{posi-differ} and \eqref{W-part}. Having \eqref{nega-differ} and \eqref{W-part}, we can prove similarly to the proof of Proposition~\ref{p:mart measure} that
\begin{equation*}
    \mathbb{E}\left[\Delta_n(k)  \,|\,  \widetilde{\mathscr{E}}^+_{a_{k-1}^n}\right]=0.
\end{equation*}
Hence $(\Delta_n(k),k\geq 1)$ is a martingale difference sequence with respect to the filtration $\big(\widetilde{\mathscr{E}}^+_{a_k^n},k\geq 1\big)$. We deduce that
\begin{equation*}
    \mathbb{E}[(f_n\cdot \cM-I(f))^2]=\sum_{k=1}^{\infty}\mathbb{E}[\Delta_n(k)^2].
\end{equation*}
By the occupation times formula, we derive that
\begin{align*}
     &\bE[({\rm I})^2] = \bE\left[\int_0^\infty h_n^2(\rho^{-,u}(t,Y^{-,u}_t),Y^{-,u}_t)\ind_{(a_{k-1}^n,a_k^n]}(Y^{-,u}_t)\ind_{\{A^{-,u}_\infty>t\}}\dd t\right]\\
     =\, &\bE\left[\int_{a_{k-1}^n}^{a_k^n}\dd y \int_0^{A^{-,u}_\infty} h_n^2(\rho^{-,u}(t,y),y)\dd_t \rho^{-,u}(t,y) \right]
    = \bE\left[\int_{a_{k-1}^n}^{a_k^n}\dd y \int_0^{\rho^{-,u}(A^{-,u}_\infty,y)}h_n^2(\ell,y)\dd \ell\right] ;\\ 
     & \bE[({\rm II})^2] = \bE\big[\bE[ ({\rm II})^2 |\,  \rB_\infty\, ] \big] = \bE\left[\int_{a_{k-1}^n}^{a_k^n}\dd y\int_{\rho(\infty,y)}^{\infty}h_n^2(\ell,y)\dd \ell\right] ;\\
    & \bE[({\rm I})({\rm II})] = \bE\big[\bE[ ({\rm I})({\rm II}) |\, \rB_\infty \,]\big] = 0.
\end{align*}
 We have $\rho^{-,u}(A^{-,u}_\infty,y) = \rho(\infty,y)$ since $\rho(t,y)=\kappa(L_t^y)$. Therefore
\[
\mathbb{E}[(f_n\cdot \cM-I(f))^2] = \Vert f_n-f\Vert_{\rL^2}^2 \to 0,\quad n\to\infty.
\]
We proved $f\cdot \cM=I(f)$, i.e. \eqref{eq:represenation}. 
\end{proof}

\begin{remark}\label{rmk:representation}
In the recurrent case (i.e. $\lambda=-1,\, \mu\delta\le 1$), we have $\rho(\infty,\cdot)\equiv 1$. Thus we can simply define a white noise on $[0,1]\times\bR$ via, for all $g\in L^2([0,1]\times\bR)$,
\begin{align}\label{eq:WN recur}
    \cW(g) = \int_0^\infty g(\rho(t,Y_t),Y_t)\dd\beta_t.
\end{align}
We can thus define the martingale measure $\cM$ via \eqref{eq:mart_WN} without enlarging the filtration $\rE^+$, and the representation of the stochastic integral will be simpler: for any $(\rE^+_x:x\ge r)$-predictable $f\in\rL^2$ compactly supported on $[0,1]\times\bR$ such that the process $(f(\rho(t,Y_t),Y_t): t\ge 0)$ admits a $\rB$-progressive version,
\[
f\cdot \cM = \int_0^\infty f(\rho(t,Y_t),Y_t)\ind_{[r,\infty)}(Y_t)\dd\beta_t.
\]
For $r\in\bR$, we can also define the martingale measure $\widetilde\cM$ by setting $\widetilde \cM_x(A):=\cW(A\times[-x,-r])$ for $x\ge r$ and $A\in\cL$, which, in the same way as in the proof of Proposition~\ref{p:mart measure}, turns out to be adapted with respect to the filtration $(\rE^-_{-x}:x\ge r)$. Following the proof of Proposition~\ref{p:representation}, we can deduce an analog of \eqref{eq:represenation}, i.e. for any $(\rE^-_{-x}:x\ge r)$-predictable $f\in\rL^2$ compactly supported on $[0,1]\times\bR$  such that the process $(f(\rho(t,Y_t),Y_t): t\ge 0)$ admits a $\rB$-progressive version,
\[
f\cdot \widetilde{\cM} = \int_0^\infty f(\rho(t,Y_t),Y_t)\ind_{(-\infty,-r]}(Y_t)\dd\beta_t.
\]
\end{remark}


We now prove the theorems stated in Section~\ref{s:introduction}.

\begin{proof}[Proof of Theorem~\ref{thm:rk}]
Recall the definition of $\tau_a(x)$ in \eqref{eq:tau_a} and note that $\{\tau_a(x)<\infty\}\in\rE_x^+$. We apply \eqref{eq:Tanaka-} with $t=\tau_a(x)$ and derive that for every $y\ge x$, on the event $\{\tau_a(x)<\infty\}$, 
\begin{align}\label{eq:proof thm}
    0 = \cR^-_{\tau_a(x)}(y) = y^- - \int_0^{\tau_a(x)} \ind_{\{Y_s\ge y\}}\dd \beta_s - \delta\int_0^{\tau_a(x)} \ind_{\{Y_s\ge y\}}\dd s - \frac{\mu-1}{\mu} (M_{\tau_a(x)}(X)\wedge y^-) + \frac{1}{2}L_{\tau_a(x)}^y.
\end{align}
If $y<0$, then $0< -y\le -x$ hence $T_{-y}(X) \le T_{-x}(X)\le \tau_a(x)$ and thus $M_{\tau_a(x)}(X)\ge -y$. If $y\ge 0$, then $M_{\tau_a(x)}(X)\wedge y^-=0$. We deduce that $M_{\tau_a(x)}(X)\wedge y^- = y^-$ on the event $\{\tau_a(x)<\infty\}$. 
Recall that $\rho(t,x)=\kappa(L_t^x)$. Applying \eqref{eq:proof thm}, we derive that for every $0\ge y\ge x$,
\begin{align}
   L_{\tau_a(x)}^{y} &= a - 2\int_0^{\tau_a(x)} \ind_{\{x\le Y_s<y\}} \dd \beta_s - 2\delta \int_0^{\tau_a(x)}\ind_{\{x\le Y_s<y\}} \dd s + \frac2\mu (y-x)\notag \\
    &= a - 2\int_0^{\tau_a(x)}\ind_{\{x\le Y_s<y\}} \dd \beta_s - 2\delta \int_x^y \kappa(L_{\tau_a(x)}^z)\dd z + \frac2\mu (y-x),\label{eq:L_tau x<0}
\end{align}
where the second equality follows from Proposition~\ref{p:biconti L} (ii). Similarly, we have for $y>x>0$,
\begin{align}\label{eq:L_tau x>0}
    L_{\tau_a(x)}^{y} = a - 2\int_0^{\tau_a(x)}\ind_{\{x\le Y_s<y\}} \dd \beta_s - 2\delta \int_x^y \kappa(L_{\tau_a(x)}^z)\dd z.
\end{align}

We now show that $L_{\tau_a(x)}^y$ is measurable with respect to $\rE_y^+$ for $y\ge x$. 
Recall the definition of  $A^{+,x}_t$ in \eqref{eq:A+}. On the event $\{\tau_a(x)<\infty\}$, it suffices to prove that $A^{+,y}_{\tau_a(x)}$ is $\rE_y^+$-measurable, since by Proposition~\ref{p:biconti L} (ii),
\[
\kappa(L^y_{\tau_a(x)}) = \rho({\tau_a(x)},y) = \lim_{z\uparrow y}\frac{1}{y-z}\int_0^{\tau_a(x)}\ind_{\{z\le Y_s<y\}} \dd s 
= \lim_{z\uparrow y}\frac{A^{+,y}_{\tau_a(x)}-A^{+,z}_{\tau_a(x)}}{y-z}.
\]
We observe that $Y_{\tau_a(x)}=x<y$ and thus $\tau_a(x)$ is a continuous and strictly increasing time of $A^{+,y}$. The desired measurability follows since for any $t>0$,
\[
\{A^{+,y}_{\tau_a(x)}\ge t\} = \{\tau_a(x)\ge \alpha^{+,y}_t\} = \{L_{\alpha^{+,y}_t}^x\le a\}= \{\rho^{+,y}(t,x)\le \kappa(a)\}.
\]

Let $f(\ell,z)=\ind_{\{0\leq \ell\leq \rho({\tau_a(x)},z),\ x\leq z<y\}}$ for $x\le y$. We have $f(\rho(t,Y_t),Y_t)=\ind_{\{0\leq t\leq \tau_a(x),\ x\leq Y_t<y\}}$ for Lebesgue-a.e. $t\ge 0$, since 
\begin{equation*}
    \int_0^\infty \ind_{\{\rho(t,Y_t)\leq \rho({\tau_a(x)},{Y_t}),\ t>\tau_a(x)\}}\dd t=\int_{\mathbb{R}}\dd y\int_{\tau_a(x)}^{\infty}\ind_{\{\rho(t,y)\leq \rho({\tau_a(x)},y)\}}\dd_t \rho(t,y)=0.
\end{equation*}
It follows that $f(\rho(t,Y_t),Y_t)$ admits a $\rB$-progressive version. We now prove that
\begin{align}\label{eq:representation ind}
    \int_x^y\cW([0,\rho({\tau_a(x)},z)],\dd z) = \int_0^{\tau_a(x)} \ind_{\{x\le Y_t<y\}}\dd \beta_t.
\end{align}
In the case $\lambda=-1$, we have $\rho({\tau_a(x)},z)\le 1$ for all $a\ge 0,\, z\ge x$. Hence \eqref{eq:representation ind} follows directly by Proposition~\ref{p:representation}.
Suppose now $\lambda=1$. Note that $f\notin\rL^2$ in general so we cannot directly apply Proposition~\ref{p:representation}. Take $m\in\bN_+$ and $f_m(\ell,z)=\ind_{\{0\leq \ell\leq \rho({\tau_a(x)},z)\wedge m,\ x\leq z<y\}}$. Then we have $f_m\in\rL^2$ and $f_m\uparrow f$ as $m\to \infty$. Hence by Proposition~\ref{p:representation},
\[
\int_x^y\cW([0,\rho({\tau_a(x)},z)\wedge m],\dd z) = \int_0^{\tau_a(x)} \ind_{[0,m]}(\rho(t,Y_t))\ind_{[x,y)} (Y_t)\dd \beta_t.
\]
We derive \eqref{eq:representation ind} by taking $m\to \infty$ and applying a standard localizing argument. To be specific, we obtain by isometry and dominated convergence theorem that, for any $T>0$,
\begin{align*}
   &\bE\left[ \left(\int_0^{\tau_a(x)\wedge T} \ind_{(m,\infty)}(\rho(t,Y_t))\ind_{[x,y)} (Y_t)\dd \beta_t\right)^2\right] \to 0 \quad \text{ as } m\to \infty.
\end{align*}
But for every $\varepsilon>0$, by Chebyshev's inequality and isometry,
\begin{align*}
    &\bP\left[ \left\vert\int_0^{\tau_a(x)} \ind_{(m,\infty)}(\rho(t,Y_t))\ind_{[x,y)} (Y_t)\dd \beta_t\right\vert > \varepsilon\right]\\
   \le \ &\frac1{\varepsilon^2} \bE\left[ \int_0^{\tau_a(x)\wedge T} \ind_{(m,\infty)}(\rho(t,Y_t))\ind_{[x,y)} (Y_t)\dd t\right] + 
   \bP\left( \sup_{t\ge T} Y_t \ge x \right).
\end{align*}
Note that $Y_t\to -\infty$ as $t\to \infty$. Therefore by letting $m\to \infty$ and then $T\to \infty$, we deduce that
\[
\int_0^{\tau_a(x)} \ind_{[0,m]}(\rho(t,Y_t))\ind_{[x,y)} (Y_t)\dd \beta_t \xrightarrow{\bP} \int_0^{\tau_a(x)} \ind_{[x,y)} (Y_t)\dd \beta_t \quad \text{ as } m\to \infty.
\]
Using the localizing sequence $({\bf x}_n:n\ge 1)$ of stopping times defined via \eqref{eq:stopping}, we can deduce in a similar way that
\[
\int_x^{y\wedge{\bf x}_n} \cW([0,\rho({\tau_a(x)},z)\wedge m],\dd z) \xrightarrow{L^2} \int_x^{y\wedge{\bf x}_n}\cW([0,\rho({\tau_a(x)},z)],\dd z) \quad \text{ as } m\to \infty,
\]
and therefore
\[
\int_x^y \cW([0,\rho({\tau_a(x)},z)\wedge m],\dd z) \xrightarrow{\bP} \int_x^y \cW([0,\rho({\tau_a(x)},z)],\dd z) \quad \text{ as } m\to \infty.
\]
We proved \eqref{eq:representation ind}. Equation \eqref{eq:SDE L_tau} holds for fixed $a\ge 0$ and $y>x$ from \eqref{eq:L_tau x<0}--\eqref{eq:representation ind}, and therefore also holds for all $y>x$ almost surely since $L_{\tau_a(x)}^y$ is continuous in $y$. 
Recall the discussion at the end of Section~\ref{s:WN}. The process $y\mapsto -\int_x^y\cW([0,\rho({\tau_a(x)},z)],\dd z)$ is a continuous $(\widetilde\rE_y^+:y\ge x)$-local martingale with quadratic variation process $\int_x^y\kappa(L_{\tau_a(x)}^z)\dd z$. Therefore by the Dambis-Dubins-Schwarz theorem, there exists a $(\widetilde\rE_y^+:y\ge x)$-Brownian motion $(W_y:y\ge x)$ such that $-\int_x^y\cW([0,\rho({\tau_a(x)},z)],\dd z) = \int_x^y \sqrt{\kappa(L_{\tau_a(x)}^z)}\dd W_z$. The generator \eqref{eq:generator L_tau} follows.

In the transient case, we notice that \eqref{eq:L_tau x<0} and \eqref{eq:L_tau x>0} still hold if we replace $\tau_a(x)$ by $\infty$ and that $\rho(\infty,x)$ is $\rE_x^+$-measurable.  Let $y\ge x$ and $\widehat{f}(\ell,z)=\ind_{\{0\leq \ell\leq \rho(\infty,z),\ x\leq z<y\}}$. Then $\widehat{f}(\rho(t,Y_t),Y_t)=\ind_{\{x\leq Y_t<y\}}$. 
We can repeat the argument on $\widehat f$ and thus deduce an analog of \eqref{eq:representation ind}, which reads
\begin{align}
    \int_x^y\cW([0,\rho(\infty,z)],\dd z) = \int_0^\infty \ind_{\{x\le Y_t<y\}}\dd \beta_t.
\end{align}
Therefore \eqref{eq:SDE L_infty} follows. Again by the Dambis-Dubins-Schwarz theorem, we deduce that there exists a $(\widetilde\rE_y^+:y\ge x)$-Brownian motion $(\widehat{W}_y: y\ge x)$ such that
\[
L_{\infty}^y = L_\infty^x + 2\int_x^y \sqrt{\kappa(L_\infty^z)}\dd \widehat{W}_z + \frac 2\mu (x^- - y^- ) - 2\delta \int_x^y \kappa(L_{\infty}^z)\dd z,\qquad y\ge x.
\]
The Brownian motion $\widehat{W}$ is independent of $\widetilde\rE_x^+$, hence of $\rE_x^+$. Since $x$ is arbitrary, the generators in \eqref{eq:generator L_infty} and the conditional independence follow. It remains to show the distribution of $L_\infty^0$. By It\^o's formula, the process $\big(\exp(-L_\infty^{-x}): x\ge 0\big)$ has the infinitesimal generator
\begin{align*}
\begin{cases}
2\ell(1-\ell)\dfrac{\dd^2}{\dd \ell^2} + 2\left( (1+\delta)(1-\ell) - \dfrac1\mu \ell \right) \dfrac{\dd}{\dd\ell}, & \lambda = 1,\\
2\ell^2(1-\ell)\dfrac{\dd^2}{\dd \ell^2} + 2\ell\left( (1+\delta)(1-\ell) - \dfrac1\mu  \right) \dfrac{\dd}{\dd\ell}, & \lambda = -1.
\end{cases}
\end{align*}
It follows by \cite[Proposition 4.9.2]{EthierKurtz1986} and elementary computation that ${\rm beta}(\delta+1,\frac1\mu)$, resp. ${\rm beta}(\delta-\frac1\mu,\frac1\mu)$ is the stationary distribution of this process in the case $\lambda=1$, resp. $\lambda=-1$. Hence it suffices to show that $L_\infty^x \eqdis L_\infty^0$ for every $x<0$. We observe that $L_\infty^x$ plays the role of $L_\infty^0$ when we replace $X$ by $X_{T_{-x}(X)+\cdot} + x$ in \eqref{eq:bbflow}. The desired equality in law then follows since the two driving processes have the same law by the strong Markov property of $X$ at time $T_{-x}(X)$. 
\end{proof}


\begin{proof}[Proof of Theorem~\ref{thm:rk recur}]
In the recurrent case, we apply \eqref{eq:Tanaka+} with $t=\tau_a(x)$ and derive that for every $y\le x$,
\[
0 = y^+ + \int_0^{\tau_a(x)} \ind_{\{\cR_s(y)>0\}} \dd \beta_s + (\delta-1) \int_0^{\tau_a(x)} \ind_{\{\cR_s(y)>0\}} \dd s +\frac{\mu-1}{\mu}[(M_{\tau_a(x)}(X) - y^-)\vee 0] + \frac12 L_{\tau_a(x)}^y.
\]
Therefore, for $x\ge y\ge 0$, 
\begin{align*}
L_{\tau_a(x)}^y &= a + 2(x-y) + 2\int_0^{\tau_a(x)} \ind_{\{y\le Y_s<x\}}\dd\beta_s + 2(\delta-1)\int_0^{\tau_a(x)} \ind_{\{y\le Y_s<x\}}\dd s\\
&= a + 2\int_0^{\tau_a(x)} \ind_{\{y\le Y_s<x\}}\dd\beta_s + 2\int_y^x \big[1+(\delta-1)\kappa(L_{\tau_a(x)}^z)\big]\dd z,
\end{align*}
where the second equality follows by Proposition~\ref{p:biconti L}.
On the other hand, we observe that for all fixed $y<0$, almost surely, $L_{\tau_a(x)}^y=0$ if and only if $T_{-y}(X)\ge \tau_a(x)$, which is also equivalent to $M_{\tau_a(x)}(X)\le y^-$. It follows that for $y\le x<0$, on the event $\{L_{\tau_a(x)}^y>0\}$,
\begin{align*}
L_{\tau_a(x)}^y &= a + 2\int_0^{\tau_a(x)} \ind_{\{y\le Y_s<x\}}\dd\beta_s + 2(\delta-1)\int_0^{\tau_a(x)} \ind_{\{y\le Y_s<x\}}\dd s + 2\bigg(1-\frac1\mu\bigg)(x-y)\\
&= a + 2\int_0^{\tau_a(x)} \ind_{\{y\le Y_s<x\}}\dd\beta_s + 2\int_y^x \bigg[\bigg(1-\frac1\mu\bigg)+(\delta-1)\kappa(L_{\tau_a(x)}^z)\bigg]\dd z.
\end{align*}
We can now follow the same idea in the proof of Theorem~\ref{thm:rk} and obtain that $A^{-,y}_{\tau_a(x)}$, hence $L_{\tau_a(x)}^y$, is measurable with respect to $\rE^-_y$. It then follows from Remark~\ref{rmk:representation} that
\[
\int_y^x \cW([0,\kappa(L_{\tau_a(x)}^z)],\dd z)=\int_y^x \cW([0,\rho({\tau_a(x)},z)],\dd z) = \int_0^{\tau_a(x)} \ind_{\{y\le Y_t<x\}}\dd \beta_t.
\]
This equation holds whenever $y\ge - M_{\tau_a(x)}(X)$; if $y< - M_{\tau_a(x)}(X)$, then $\cR(y)<0$ in $[0,\tau_a(x)]$ from \eqref{eq:bbflow}, hence $L_{\tau_a(x)}^y\equiv 0$. 
Therefore \eqref{eq:SDE L_tau y<x} holds, and generator \eqref{eq:generator L_tau y<x} follows from the discussion in Section~\ref{s:WN}.
\end{proof}


\section{Connection to the Jacobi flow and the burglar}
In this section, we present the concept of Jacobi flows introduced in \cite{aidekon2023stochastic} and the link with the Bass--Burdzy flow driven by a perturbed drifted Brownian motion. This link will be used to show the connection between the flows and the burglar defined in \eqref{eq:Z10}.

\subsection{Connection to the Jacobi flow}\label{s:Jacobi}

\begin{definition}[{\cite[Definitions 5.1 \& 5.3]{aidekon2023stochastic}}]\label{def:Jacobi}
Let $\cW$ be a white noise. Define the Gaussian noise $\cJ$ by setting
\[
\cJ(A\times[x,y]) := \cW(A\times[x,y]) - |A|\, \cW([0,1]\times[x,y]),\quad x\le y,\, A\in\cB([0,1]).
\]
Let $\delta>0,\,\delta'\in\bR$ and $\cY=(\cY_{r,x}(v):x\ge r)_{r\in\bR,\ v\in[0,1]}$ be a collection of continuous processes. We call $\cY$ a Jacobi($\delta,\delta'$) flow (killed at 1) if 
\begin{enumerate}[(1)]
    \item For each $v\in[0,1],\ r\in\bR$, the process $(\cY_{r,x}(v):x\ge r)$ is almost surely the unique strong solution to the equation
    \begin{align}\label{eq:Jacobi SDE}
        \cY_{r,x}(v) &= v + 2\int_r^x \cJ([0,\cY_{r,y}(v)],\dd y) + \int_r^x \big( \delta(1-\cY_{r,y}(v)) - \delta' \cY_{r,y}(v) \big) \dd y,\qquad x\ge r.
    \end{align}
    The process is absorbed when hitting 1.
    
    \item Almost surely, the following regularity conditions hold: 
    \begin{enumerate}[(i)]
        \item for all $v\in[0,1], r\in\bR$, $\cY_{r,r}(v)=v$,
        \item for all $x\ge r$, the function $v\mapsto \cY_{r,x}(v)$ is c\`adl\`ag,
        \item for all $r\le r'$ and $v,v'\in[0,1]$, if $\cY_{r,r'}(v)>v'$, resp. $\cY_{r,r'}(v)<v'$, then for all $x\ge r'$, $\cY_{r,x}(v)\ge \cY_{r',x}(v')$, resp. $\cY_{r,x}(v)\le \cY_{r',x}(v')$.
        \item for all $r\le x\le y$ and $v\in [0,1]$, if $\cY_{r,x}(v)=1$ then $\cY_{r,y}(v)=1$.
    \end{enumerate}
\end{enumerate}
\end{definition}

\begin{definition}[{\cite[Definition 5.4]{aidekon2023stochastic}}]
Let $-\infty<\delta<2$ and $\delta'\in\bR$. We call $\cY$ a killed ${\rm Jacobi}(\delta,\delta')$ flow if each flow line is a solution to \eqref{eq:Jacobi SDE} but also absorbed at 0, and $\cY$ satisfies the same regularity conditions.

\end{definition}

\begin{remark}\label{rmk:Jacobi}
\begin{enumerate}[(1)]
    
    \item The Gaussian noise $\cJ$ in Definition~\ref{def:Jacobi} is uniquely characterized as follows: it has independent increments in the vertical direction, and its covariance measure is given by, for $A,\widetilde{A}\in\cB([0,1])$ and $x\le y$,
    \begin{align}\label{eq:Jacobi variance}
        \bE[\cJ(A\times[x,y])\cJ(\widetilde{A}\times[x,y])] = (y-x)\big(|A\cap \widetilde{A}|-|A| |\widetilde{A}|\big).
    \end{align}

    \item We can define the Jacobi flow with varying parameters. Let $\overline{\delta}(x)=\delta_1\ind_{(-\infty,0)}(x) + \delta_2\ind_{[0,\infty)}(x)$ and $\overline{\delta'}(x)=\delta'_1\ind_{(-\infty,0)}(x) + \delta'_2\ind_{[0,\infty)}(x)$ with $\delta_1,\delta_2\ge 0$. We  call $\cY$ a $\mathrm{Jacobi}(\overline{\delta},\overline{\delta'})$ flow or a $\rm Jacobi (\delta_1\,|\,\delta_2,\, \delta'_1\,|\,\delta'_2)$ flow if $\cY$ satisfies all the regularity conditions, but for each $v\in[0,1],\ r\in\bR$, the process $(\cY_{r,x}(v):x\ge r)$ is almost surely the unique strong solution to the equation
    \begin{align}\label{eq:Jacobi vary SDE}
        \cY_{r,x}(v) &= v + 2\int_r^x \cJ([0,\cY_{r,y}(v)],\dd y) + \int_r^x \big( \overline{\delta}(y)(1-\cY_{r,y}(v)) - \overline{\delta'}(y) \cY_{r,y}(v) \big) \dd y , \quad x\ge r,
    \end{align}
     absorbed when hitting 1. By \cite[Theorem 5.9]{aidekon2023stochastic}, the Jacobi flow can be constructed through a transformation of the BESQ flow. Therefore, in the spirit of \cite[Lemma 2.16]{aidekon2024meetingsquaredbesselflow}, a $\rm Jacobi(\overline{\delta},\overline{\delta'})$ flow can be constructed by gluing a $\rm Jacobi(\delta_1,\delta'_1)$ flow with an independent $\rm Jacobi(\delta_2,\delta'_2)$ flow at level 0, so that the perfect flow property still holds: with probability 1, $\cY_{r,y}=\cY_{x,y}\circ \cY_{r,x}$ for all $r\le x\le y$.
\end{enumerate}
\end{remark}

Assume now $\lambda=1$ in \eqref{eq:bbflow}. Recall that $L_t^x$ is the semimartingale local time at level $0$ of the process $\cR(x)$ and $\widehat{L}_t^x = 1 - \exp(-L_t^x)$ in \eqref{def:Lhat}. We define for $v\in[0,1)$ and $x\ge r$,
\begin{align}\label{eq:S & Psi}
    S_{r,x}(v):= \widehat{L}_{\tau_{\widehat v}(r)}^x = 1-\exp\big(-L_{\tau_{\widehat v}(r)}^x\big), \qquad \Psi_x:= \widehat{L}_\infty^x = 1-\exp(-L_\infty^x),
\end{align}
where $\widehat{v} := -\log (1-v)$. The following proposition states that $(S_{r,x}(a): x\ge r)_{r\in\bR,\, 0\le a\le \Psi_r}$ forms a sub-flow of a Jacobi flow with varying parameters.

\begin{proposition}\label{p:Jacobi}
Suppose $\delta\ge 0$ and $\lambda=1$. Almost surely, $\Psi_x=S_{r,x}(\Psi_r)$ for any $r\le x$, and $\Psi_x = \lim_{r\to -\infty}S_{r,x}(0)$ for any $x\in\bR$. 
The flow $S=(S_{r,x}(v): x\ge r)_{r\in\bR,\, 0\le v\le \Psi_r}$ has the same law as the restriction of a $\mathrm{Jacobi}(\frac2\mu\,|\, 0,\, 2+2\delta)$ flow to the region $\{(v,r):r\in\bR,\, 0\le v\le \Psi_r\}$.
\end{proposition}

\begin{proof}
Recall that $\kappa(\ell)=e^\ell-1$ in this setting. Applying It\^o's formula to \eqref{eq:SDE L_tau}, we deduce that on the event $\{v<\Psi_r\} = \{\tau_{\widehat{v}}(r)<\infty\}$,
\begin{align*}
    S_{r,x}(v) &= 1-\exp\big(-L_{\tau_{\widehat{v}}(r)}^x\big)\\
    &= v +\int_r^x \exp\big(-L_{\tau_{\widehat v}(r)}^y\big)\dd_y L_{\tau_{\widehat{v}}(r)}^y - \frac12 \int_r^x \exp\big(-L_{\tau_{\widehat v}(r)}^y\big)\dd_y \langle L_{\tau_{\widehat{v}}(r)}^\cdot\rangle_y\\
    &= v - 2\int_r^x \exp\big(-L_{\tau_{\widehat v}(r)}^y\big) \cW\big(\big[0,\exp\big(L_{\tau_{\widehat{v}}(r)}^y\big)-1\big],\dd y\big) \\
    &\quad + \int_r^x  \left[ \frac2\mu \ind_{[-r^-,0]}(y)\cdot \exp\big(-L_{\tau_{\widehat v}(r)}^y\big) - (2+2\delta)\Big(1- \exp\big(-L_{\tau_{\widehat v}(r)}^y\big)\Big)\right]\dd y.
\end{align*}
We define $\cJ$ via \eqref{def:M} and extend it to an $L^2$-signed measure via standard arguments. Then $\cJ$ is a Gaussian noise with independent increments in the vertical direction and the same covariance measure as \eqref{eq:Jacobi variance}. On the event $\{v<\Psi_r\}$, we can rewrite the equation as
\[
S_{r,x}(v) = v + 2\int_r^x \cJ([0,S_{r,y}(v)],\dd y) + \int_r^x \left(\frac2\mu \ind_{[-r^-,0]}(y)(1-S_{r,y}(v))-(2+2\delta) S_{r,y}(v)\right)\dd y.
\]
The process $(S_{r,x}(v):x\ge r)$ is absorbed at 0 for $x\ge r^+$ since $\rho(\tau_{\widehat{v}}(r),x) = \exp(L_{\tau_{\widehat{v}}(r)}^x)-1$, interpreted as the local time of $Y=\cR^{-1}(0)$, has the same property. Moreover, flow lines in $S$ cannot hit 1 because $\rho(\infty,x)<\infty$ for all $x\in\bR$. The regularity conditions (i), (ii), (iii) follow directly from properties of $\rho$ as the local time process of $Y$.
The fact that $\Psi_x=S_{r,x}(\Psi_r)$ for any $r\le x$ follows by a similar application of It\^o's formula to \eqref{eq:SDE L_infty}. 
To finish the proof, let us prove that for all $x\in\bR$, $S_{r,x}(0)=\Psi_x$ for every $r<x$ sufficiently small. We have $Y_t\to -\infty$ as $t\to \infty$ by Proposition~\ref{l:limit} (i). Define $t_x:=\sup\{t\ge 0: Y_t =x\}$ where $\sup \emptyset := -\infty$ by convention. If $t_x$ is finite, then $\rho(\tau_0(r),x)=\rho(\infty,x)$ for all $r< \inf_{u\in[0,t_x]}Y_u$. If $t_x=-\infty$, then in particular $x>0$ and therefore $\rho(\tau_0(r),x)=\rho(\infty,x)=0$ for all $r\le 0$. The proof is thus complete.
\end{proof}

\begin{corollary}\label{c:Xi}
$\Psi_0$ has the law of $\rm beta(\frac1\mu,1+\delta)$. Conditionally on $\Psi_0=r$, $(\Psi_x:x\ge 0)$ and $(\Psi_{-x}:x\ge 0)$ are independent processes starting at $r$ with laws respectively $\rm Jacobi(0,2+2\delta)$ and $\rm Jacobi(\frac2\mu,2+2\delta)$.
\end{corollary}
\begin{proof}
It is a direct consequence of Theorem~\ref{thm:rk} and Proposition~\ref{p:Jacobi}. 
\end{proof}

We present a connection between the flow $S$ and the process $Y$. Recall the time-change $\zeta$ defined via \eqref{eq:zeta} and the process $\widehat{Y}_t:=Y_{\zeta^{-1}(t)}$. 

\begin{lemma}\label{l:zeta}
$\zeta(\infty):=\lim_{t\to \infty}\zeta(t)=\infty$.
\end{lemma}
\begin{proof}
By the occupation times formula applied to $Y$,
\begin{align*}
    \zeta(\infty) &= \int_0^\infty \frac{\dd s}{\big(1 + \rho(s,Y_s)\big)^2}= \int_{\bR}\dd x \int_0^\infty \frac{\dd_s \rho(s,x)}{\big(1 + \rho(s,x)\big)^2}= \int_\bR \dd x \frac{\rho(\infty,x)}{1+\rho(\infty,x)} = \int_\bR \Psi_x\dd x.
\end{align*}
But by Corollary~\ref{c:Xi}, $(\Psi_{-x}:x\ge 0)$ is a $\rm Jacobi(\frac2\mu,2+2\delta)$ process with initial distribution $\rm beta(\frac1\mu,1+\delta)$, and hence is stationary and ergodic. We deduce that
\[
\lim_{y\to \infty} \frac1y \int_0^y \Psi_{-x}\dd x = \bE[\Psi_0] = \frac{1}{1+\mu(1+\delta)}.
\]
The result is then an immediate consequence.
\end{proof}

By definition of $\zeta$ in \eqref{eq:zeta} and the occupation times formula, 
\begin{align*}
    &\int_0^t \ind_{\{\widehat{Y}_s<x\}}\dd s = \int_0^{\zeta^{-1}(t)} \ind_{\{Y_u<x\}} \frac{\dd u}{\big(1 + \rho\big(u,Y_u\big)\big)^2}\\
    =\ &\int_{-\infty}^x \dd y \int_0^{\zeta^{-1}(t)} \frac{\dd_u \rho(u,y)}{(1+\rho(u,y))^2} = \int_{-\infty}^x \dd y \, \frac{\rho({\zeta^{-1}(t)},y)}{1+\rho({\zeta^{-1}(t)},y)}.
\end{align*}
We deduce that 
\begin{align}\label{eq:local time Yhat}
    \cL_{\widehat{Y}}(t,x) = \frac{\rho({\zeta^{-1}(t)},x)}{1+\rho({\zeta^{-1}(t)},x)} = 1 - \exp\big(-L_{\zeta^{-1}(t)}^x\big),
\end{align}
and thus $\cL_{\widehat{Y}}(\infty,x)=\Psi_x$ since $\zeta^{-1}(\infty)=\infty$ by Lemma~\ref{l:zeta}. Let $\tau_v^r(\widehat{Y}):=\inf\{t>0:\cL_{\widehat{Y}}(t,r)>v\}$ for $r\in\bR$ and $v\in[0,1]$. Note that $\tau_v^r(\widehat{Y})$ is finite for $v\in[0,\Psi_r)$ and infinite for $v\in [\Psi_r,1]$.  By Proposition~\ref{p:biconti L}~(ii) and \eqref{eq:zeta}, $\cL_{\widehat{Y}}(t,x)$ is bicontinuous in $(t,x)$. 
Recall the definition of the forward local time flow of a process in Definition~\ref{def:flow}.

\begin{corollary}\label{c:time-change}
The forward local time flow of the process $\widehat{Y}$ is $S=(S_{r,x}(v):x\ge r)_{r\in\bR,\, v\in[0,\Psi_r]}$.
\end{corollary}
\begin{proof}
 Since $\zeta$ is strictly increasing, it is elementary to derive that $\tau_v^r(\widehat{Y}) = \zeta \big(\tau_{\widehat v}(r)\big)$. It together with \eqref{eq:S & Psi}--\eqref{eq:local time Yhat} implies that $\cL_{\widehat{Y}}(\tau_v^r(\widehat{Y}),x) = S_{r,x}(v)$ for all $x\ge r$ and $v\in[0,\Psi_r)$. Having deduced that $\cL_{\widehat{Y}}(\infty,x)=\Psi_x$ for all $x\in\bR$, we complete the proof.
\end{proof}

\subsection{Connection to the burglar}\label{s:burglar}

We now prove Theorem~\ref{thm:burglar}, i.e. the connection between the burglar and the Bass--Burdzy flow driven by the opposite of a PRBM defined in \eqref{PRBM}. Recall the inverse local time $\tau_1^0$ in \eqref{def-tau}.
Let $J$ be a Bessel process of dimension 3 starting at 0, independent of $\fM$. Note that $J$ is transient and positive in $(0,\infty)$, and $\big(L_J(\infty,x):x\ge 0\big)$ is a squared Bessel process of dimension 2 by \cite[Theorem 3.3]{yor1992some}. We define
\begin{equation}\label{Mtilde}
\widetilde{\fM}_t:= \begin{cases}
    \fM_t, & 0\le t\le \tau_1^0,\\
    J_{t-\tau_1^0},  & t>\tau_1^0,
\end{cases}
\end{equation}
and
\begin{align}\label{eq:total local time}
R_r:= L_{\widetilde{\fM}}(\infty,r)= L_\fM(\tau_1^0,r) + L_J(\infty,r),\qquad r\in\bR.
\end{align}
Note that the process $R$ has the same law as the process defined in \eqref{def:Rhat}, and that $R_x=L_\fM(\tau_1^0,x)$ for $x\le 0$. 
Define $\tau_a^r(\widetilde{\fM}):=\inf\{t>0:L_{\widetilde{\fM}}(t,r)>a\}$ for $r\in\bR$, and write 
\[
\cS_{r,x}(a):=L_{\widetilde{\fM}}(\tau_a^r(\widetilde{\fM}),x).
\]
Let $\cZ$ be the burglar $Z^{(1)}$ defined in \eqref{eq:Z10} with $R$ defined in \eqref{eq:total local time}. Recall the process $\widehat{Y}=Y_{\zeta^{-1}(\cdot)}$ with $Y$ and $\zeta$ defined in \eqref{def:Y} and \eqref{eq:zeta} respectively. 

\begin{proof}[Proof of Theorem~\ref{thm:burglar}]
We shall work with the process $\cZ$ rather than $Z^{(1)}$ for convenience.
Following \cite[Definition 5.6]{aidekon2023stochastic} and applying \cite[Theorem 5.1]{aidekon2024infinite}, the forward local time flow $\big(\cS_{r,x}(a):x\ge r\big)_{r\ge 0,\, 0\le a\le R_r}$ is a $\rm BESQ^{0,2}$ flow, and the backward local time flow $\big(\cS_{-r,-x}(a):x\ge r\big)_{r\ge 0,\, 0\le a\le R_{-r}}$ is a killed $\rm BESQ^{2-\frac2\mu,0}$ flow. These two flows are independent.
Recall that $t^*=\text{argmin}_{t\in [0,\tau_1^0]} \fM_t$.
Set $\eta_R(x) := \int_0^x \frac{\dd r}{R_r}$ for $x\ge \inf_{0\le t\le \tau_1^0} \fM_t$, and then
\begin{align*}
\cY_{r,x}(v) &:=\frac{\cS_{\eta_R^{-1}(r),\eta_R^{-1}(x)}\big(v R_{\eta_R^{-1} (r)}\big)}{R_{\eta_R^{-1} (x)}}, & r,x\in\bR, \, v\in[0,1];\\
\Xi_x &:= \frac{L_\fM(t^*,\eta_R^{-1} (x))}{ R_{\eta_R^{-1} (x)}}, & x\in\bR.
\end{align*}
By \cite[Theorem 5.9]{aidekon2023stochastic}, $(\cY_{r,x}(v):x\ge r)_{r\ge 0,\, v\in[0,1]}$ is a $\rm Jacobi(0,2)$ flow, and $(\cY_{-r,-x}(v):x\ge r)_{r\ge 0,\, v\in[0,1]}$ is a killed $\rm Jacobi(2-\frac2\mu,0)$ flow. The two flows are independent of each other, and both are independent of $R$. 
By \cite[Lemma 5.13]{aidekon2023stochastic}, the flow $(\cY_{-r,-x}(v):x\ge r)_{r\ge 0,\, v\in[0,1]}$ is the dual flow of $(\cY_{r,x}(v):x\ge r)_{r\le x\le 0,\, v\in[0,1]}$, i.e.
\[
\cY_{-r,-x}(v) = \cY^*_{r,x}(v):=\inf\{u\in[0,1]:\cY_{-x,-r}(u)>v\}\wedge 1,\qquad \forall\, x\ge r\ge 0,\, v\in[0,1].
\]
Therefore by \cite[Theorem~6.1]{yu2026continuous}, $(\cY_{r,x}(v):x\ge r)_{r\in\bR,\, v\in[0,1]}$ is a $\rm Jacobi(\frac2\mu\,|\,0,2)$ flow.  
Following the proof of \cite[Theorem 6.2]{aidekon2023stochastic}, we obtain that $\Xi_x=\cL_\cZ(\infty,x)$ and that the forward local time flow of $\cZ$ is $(\cY_{r,x}(v):x\ge r)_{r\in\bR,\, v\in[0,\Xi_r]}$. 
By \cite[Proposition 6.4]{aidekon2023stochastic}, we have $\Xi_x=\cY_{r,x}(\Xi_r)$ for $x\ge r$ and $\Xi_x = \lim_{r\to -\infty}\cY_{r,x}(0)$ for $x\in\bR$. Combining these facts and Proposition~\ref{p:Jacobi}, we derive that
\[
\left(\Xi, \, (\cY_{r,x}(v):x\ge r)_{r\in\bR,\, v\in[0,\Xi_r]}\right) \eqdis \left(\Psi, \, (S_{r,x}(v):x\ge r)_{r\in\bR,\, v\in[0,\Psi_r]}\right).
\]
The left-hand side is the local time flow of $\cZ$ by \cite[Proposition 4.3]{aidekon2023stochastic}.
The right-hand side is the local time flow of $\widehat{Y}$ by Corollary~\ref{c:time-change}. From \cite[Remarks 2.8 \& 2.9]{aidekon2023stochastic}, a continuous process admitting bicontinuous local times can be reconstructed from its forward local time flow by a (deterministic) measurable map. Therefore $\cZ$ and $\widehat{Y}$ have the same law.
Since $\cY$ is independent of $R$ and $\Xi$ is measurable with respect to $\cY$, the forward local time flow of $\cZ$, hence $\cZ$ itself, is independent of $R$.
\end{proof}

We define the time-change associated with $\cZ$ by
\[
\chi(t):= \int_0^t \frac{\dd s}{\big(1-\cL_\cZ(s,\cZ_s)\big)^2},\qquad t\ge 0.
\]
This time-change appeared in \cite[Lemma 4.5]{aidekon2023stochastic}.
Note that $\chi$ maps $[0,\infty)$ to itself since $\cL_\cZ(\infty,x)<1$ for all $x\in\bR$. We then define $\overline{\cZ}:=\cZ_{\chi^{-1}(\cdot)}$.

\begin{corollary}\label{c:time-changed burglar}
\begin{enumerate}[(i)]
    \item The processes $\overline{\cZ}$ and $Y$ have the same law.
    \item The process $H=(H_t:t\ge 0)$ defined via
    \begin{align}\label{eq:H with Z}
        H_t := -\overline{\cZ}_t - \int_{-\infty}^{\overline{\cZ}_t} \cL_{\overline{\cZ}}(t,y)\dd y
    \end{align}
    has the law of $X$, which is defined via \eqref{eq:X by BM} with $\delta=0$ and the same $\mu$. In particular, the process $\beta^H:=H-(1-\frac1\mu)M(H)$ is a standard Brownian motion.
\end{enumerate}
\end{corollary}
Define $R_t(x) = \int_{\overline\cZ_t}^x (1+\cL_{\overline\cZ}(t,y))\dd y$ for all $t\ge 0$ and $x\in\bR$. A direct consequence of Corollary~\ref{c:time-changed burglar} is that the flow $R$ has the same law as $\cR$ induced by \eqref{eq:bbflow} with $\lambda=1$ and $X=-\fM$. 
We refer to Lupu, Sabot and Tarr\`es \cite{lupu2021inverting}, A\"id\'ekon, Hu and Shi \cite{aidekon2023stochastic} and Dremaux \cite{dremaux2026generalized} for an analog of this kind of connection. 

\begin{proof}[Proof of Corollary~\ref{c:time-changed burglar}]
From \eqref{eq:local time Yhat},
\begin{equation}\label{eq:local time Y}
    \rho(t,y)= \cL_{Y}(t,y) = \frac{\cL_{\widehat{Y}}(\zeta(t),y)}{1-\cL_{\widehat{Y}}(\zeta(t),y)}.
\end{equation}
Combining \eqref{eq:local time Y} with \eqref{eq:zeta}, we derive that
\[
\zeta^{-1}(t) = \int_0^t \frac{\dd s}{(1-\cL_{\widehat{Y}}(s,\widehat{Y}_s))^2},\qquad t\ge 0,
\]
which has the same form as the definition of the time-change $\chi$. Since $Y=\widehat{Y}_{\zeta(\cdot)}$, (i) follows directly by Theorem~\ref{thm:burglar}. Substituting $x=Y_t=\cR^{-1}_t(0)$ into \eqref{eq:bbflow} (with $\lambda=1$) and applying Proposition~\ref{p:biconti L}~(ii), we deduce that
\begin{align}\label{eq:X with Y}
X_t = -Y_t -\int_0^t \ind_{\{Y_s< Y_t\}}\dd s = -Y_t -\int_{-\infty}^{Y_t} \rho(t,y)\dd y.
\end{align}
The bicontinuous local times of a continuous process are measurable with respect to the process itself by the occupation times formula. We deduce (ii) by comparing \eqref{eq:X with Y} with \eqref{eq:H with Z} and applying (i).
\end{proof}

We close this section with a discussion on how to construct the noise that drives the forward local time flow of $\cZ$ directly from $\cZ$ itself, with an independent auxiliary white noise. 
Combining Theorem~\ref{thm:burglar} and Corollary~\ref{c:time-changed burglar}, we obtain that $(\beta^H,H,\overline{\cZ},\cY|_{v\le \Xi_r})$ and $(\beta,X,Y,S|_{v\le \Psi_r})$ have the same law. 
Let $\widetilde{W}$ be a white noise independent of $\cZ$. We define a new white noise $\cW_\cZ$ via \eqref{eq:WN} with $(Y,\rho,\beta,\widetilde{\cW})$ replaced by $(\overline{\cZ}, \cL_{\overline{\cZ}},\beta^H,\widetilde{W})$, i.e. for every deterministic $g\in L^2(\bR_+\times\bR)$,
\begin{align}\label{eq:WN burglar}
    \cW_\cZ(g):= &\int_0^\infty g(\cL_{\overline{\cZ}}(t,\overline{\cZ}_t),\overline{\cZ}_t)\dd \beta^H_t + \int_{\bR_+\times\bR} g(\ell,x) \ind_{[\cL_{\overline{\cZ}}(\infty,x),\infty)}(\ell)\, \widetilde{W}(\dd\ell,\dd x).
\end{align}
Recall that $\overline\cZ_t=\cZ_{\chi^{-1}(t)}$. We have $\cL_{\overline{\cZ}}(t,x) = \frac{\cL_\cZ(\chi^{-1}(t),x)}{1-\cL_\cZ(\chi^{-1}(t),x)}$ from \eqref{eq:local time Y} and Corollary~\ref{c:time-changed burglar}. Define $\phi(\ell):=\frac{\ell}{1-\ell}$ for $0\le \ell<1$. We apply the time-change formula in the first term of \eqref{eq:WN burglar}, and thus rewrite 
\begin{align*}
    \cW_\cZ(g)= &\int_0^\infty g(\phi(\cL_\cZ(\chi^{-1}(t),\cZ_{\chi^{-1}(t)})),\cZ_{\chi^{-1}(t)})\dd \beta^H_t + \int_{\bR_+\times\bR} g(\ell,x) \ind_{[\phi(\cL_{\cZ}(\infty,x)),\infty)}(\ell)\, \widetilde{W}(\dd\ell,\dd x)\\
    = & \int_0^\infty \frac{g(\phi(\cL_\cZ(u,\cZ_u)),\cZ_u)}{1-\cL_\cZ(u,\cZ_u)} \dd \beta^{H,\chi}_u + \int_{\bR_+\times\bR} g(\ell,x)\ind_{[\phi(\cL_{\cZ}(\infty,x)),\infty)}(\ell)\, \widetilde{W}(\dd\ell,\dd x),
\end{align*}
where 
\[
\beta^{H,\chi}_t:= \int_0^{\chi(t)} \big(1-\cL_\cZ(\chi^{-1}(s),\cZ_{\chi^{-1}(s)})\big)\dd \beta^H_s
\]
is also a Brownian motion by L\'evy's characterization. Define the noise $\cJ_\cZ$ via \eqref{def:M} with $\cW_\cZ$ playing the role of $\cW$.
Then combining Proposition~\ref{p:Jacobi} and Corollary~\ref{c:time-change}, we obtain the following consequence.

\begin{corollary}
Almost surely, $\Xi_x=\cY_{r,x}(\Xi_r)$ for any $r\le x$, and $\Xi_x = \lim_{r\to -\infty}\cY_{r,x}(0)$ for any $x\in\bR$. 
The forward local time flow of $\cZ$, i.e. $(\cY_{r,x}(v),\, x\ge r)_{r\in\bR,\, v\in[0,\Xi_r]}$, is the $\mathrm{Jacobi}(\frac2\mu\,|\, 0,\, 2)$ flow \textbf{driven by $\cJ_\cZ$} and restricted to the region $\{(v,r):r\in\bR,\, 0\le v\le \Xi_r\}$.     
\end{corollary}

\appendix

\section{Stochastic flow of diffeomorphisms}\label{s:diffeomorphism}
In this section, we refer to \cite[Section 5]{hu-warren} and prove that \eqref{eq:change bbflow} induces a flow of diffeomorphisms in the setting of \eqref{eq:sigma & X}. 
We begin by approximating of $\sigma$ by $\sigma_n\in C^1(\bR)$ with $\sigma_n$ globally Lipschitz, in which case \eqref{eq:change bbflow} (with $\sigma$ replaced by $\sigma_n$) induces a flow of diffeomorphisms $(\cX^{(n)}_t(x):t\ge 0,\, x\in\bR)$ by \cite[Theorem V.49]{protter2005stochastic}. Then we prove that $\partial \cX^{(n)}_t(x)/\partial x$, as a function of $x$, will converge in the sense of distributions to a positive continuous limit, which we identify as the derivative of $\cX_t(x)$ with respect to $x$.

Let $\rF=(\rF_t: t\ge 0)$ be a filtration. For a continuous $\rF$-adapted process $Z$ and an $\rF$-stopping time $T$, let $Z^T$ denote the process defined by $Z^T_t:=Z_{T\wedge t}$ for all $t\ge 0$. If $Z$ is an $\rF$-semimartingale, we can define for all $p\in [1,\infty]$ its $\cS^p$ and $\cH^p$ norms as
\begin{align*}
    \Vert Z \Vert_{\cS^p} = \bigg\Vert \sup_{t\ge 0} |Z_t| \bigg\Vert_{L^p}, \quad
    \Vert Z \Vert_{\cH^p} = \bigg\Vert |Z_0| + \langle M \rangle_\infty^{1/2} + \int_0^{\infty}|\dd A_s| \bigg\Vert_{L^p},
\end{align*}
where $M$ is a continuous local martingale and $A$ a continuous $\rF$-adapted process of finite variation, satisfying $M_0=A_0=0$ and $Z_0+M+A=Z$. This decomposition of $Z$ is unique; see for example \cite[p.358]{rogers2000diffusions}. The $\cH^p$ norm is stronger than the $\cS^p$ norm  for $p\in[1,\infty)$; see \cite[Theorem V.2]{protter2005stochastic}. Moreover, every continuous semimartingale $Z$ with bounded  $Z_0$ is locally in the space $\cH^\infty$, meaning there exists an increasing sequence of $\rF$-stopping times $(T_k: k\ge 1)$ tending to $\infty$ such that $Z^{T_k}$ is in $\cH^\infty$.

In this section, we will assume that $\rF=\rB$, i.e. the right-continuous filtration generated by the Brownian motion $\beta$. Recall the definition of $\tX$ in \eqref{eq:sigma & X}. For every $k\ge 1$, we define the $\rB$-stopping time
\begin{align}\label{eq:stopping tildeX}
T_k:=\inf\{t>0: \sqrt t +|1-\mu|L_U(t,0)=k\},
\end{align}
so that $T_k\le k^2$, $\Vert \tX^{T_k}\Vert_{\cH^\infty}\le k$ and $T_k\to \infty$ as $k\to \infty$. 

\begin{lemma}\label{l:bounded H norm}
For all $T\ge 0$, $k\ge 1$ and $p>2$, $\sup_{x\in[-T,T]} \Vert (\cX(x))^{T_k} \Vert_{\cH^p} <\infty$.
\end{lemma}
\begin{proof}
We first deal with the case $x=0$ and briefly write $\cX_t:=\cX_t(0)$. From \eqref{eq:change bbflow} and It\^o's formula,
\[
|\cX_t|^p = p\int_0^t |\cX_u|^{p-2}\cX_u\sigma(\cX_u)\dd \tX_u + \frac{p(p-1)}2 \int_0^t |\cX_u|^{p-2} \sigma(\cX_u)^2\dd u.
\]
Set $T_{n,k}:=T_k\wedge\inf\{t>0:
|\cX_t|=n\}$ for $n,k\ge 1$. Then for every $\rB$-stopping time $\tau$, 
\begin{align*}
&\bE\left[\sup_{0\le t\le \tau\wedge T_{n,k}}|\cX_t|^p\right]\\
\le\ & p\,\bE\left[\sup_{0\le t\le \tau\wedge T_{n,k}} \bigg|\int_0^{t} |\cX_u|^{p-2}\cX_u\sigma(\cX_u)\dd \beta_u\bigg|\right] 
+ p\, \bE\left[\int_0^{\tau\wedge T_{n,k}} |\cX_u|^{p-1} |\sigma(\cX_u)|\cdot |1-\mu|\dd L_U(u,0) \right]\\
& + \frac{p(p-1)}{2}\,  \bE\left[ \int_0^{\tau\wedge T_{n,k}} |\cX_u|^{p-2} \sigma(\cX_u)^2 \dd u \right].
\end{align*}
By BDG's inequality (cf. \cite[Theorem IV.4.1]{revuz2013continuous}), there exists some $C_p>0$ such that
\begin{equation}\label{eq:est mart}
\begin{aligned}
&\bE\left[\sup_{0\le t\le \tau\wedge T_{n,k}} \bigg|\int_0^{t} |\cX_u|^{p-2}\cX_u\sigma(\cX_u)\dd \beta_u\bigg|\right] 
\le C_p \bE\left[\bigg(\int_0^{\tau\wedge T_{n,k}} |\cX_u|^{2p-2}\sigma(\cX_u)^2\dd u\bigg)^{\frac 12}\right]\\
\le\, & C_p \bE\left[ \sup_{0\le t\le \tau\wedge T_{n,k}}\Big|\cX_t\Big|^{\frac p2}\bigg(\int_0^{\tau\wedge T_{n,k}} |\cX_u|^{p-2}\sigma(\cX_u)^2\dd u\bigg)^{\frac 12}\right]\\
\le\, & C_p \left(\varepsilon\, \bE\left[ \sup_{0\le t\le \tau\wedge T_{n,k}}\Big|\cX_t\Big|^p \right] + \frac{1}{4\varepsilon} \bE\left[\int_0^{\tau\wedge T_{n,k}} |\cX_u|^{p-2}\sigma(\cX_u)^2\dd u\right]\right)
\end{aligned}
\end{equation}
for all $\varepsilon>0$. Observe that $|\sigma(x)|\le (2+2\delta)(1+|x|)$ is of linear growth. It follows that
\begin{align}\label{eq:X bigger}
|\cX_u|^{p-2}\sigma(\cX_u)^2 \le (2+2\delta)^2 |\cX_u|^{p-2} (1+|\cX_u|)^2\le 4(2+2\delta)^2 (1+|\cX_u|^p).
\end{align}
Moreover, from the definition of $T_k$ in \eqref{eq:stopping tildeX}, it follows that 
\begin{equation}\label{eq:est FV}
\begin{aligned}
&\bE\left[\int_0^{\tau\wedge T_{n,k}} |\cX_u|^{p-1} |\sigma(\cX_u)|\cdot |1-\mu|\dd L_U(u,0) \right] \\
\le\, & (2+2\delta)\bE\left[\int_0^{\tau\wedge T_{n,k}} |\cX_u|^{p-1} (1+|\cX_u|) \cdot |1-\mu|\dd L_U(u,0)\right]\\
\le\, & 2(2+2\delta) \left(k+\bE\left[\int_0^{\tau\wedge T_{n,k}} |\cX_u|^p  \cdot |1-\mu|\dd L_U(u,0)\right] \right).
\end{aligned}
\end{equation}
We set
\[
u^{(n,k)}(t)=\sup_{0\le s\le t}\Big|\cX_s^{T_{n,k}}\Big|^p.
\]
Taking $\varepsilon=1/(2p\, C_p)$ and combining \eqref{eq:est mart}--\eqref{eq:est FV}, we deduce that there exists some $C>0$ depending only on $\delta,\mu,k,p$ (hence independent of $n$) such that for every $\rB$-stopping time $\tau$,
\[
\bE[u^{(n,k)}(\tau)]\le C\Bigg(1+\bE\left[\int_0^\tau u^{(n,k)}(t) \, \dd_t \big(t\wedge T_k+|1-\mu|L_U(t\wedge T_k,0)\big) \right]\Bigg).
\]
We have $u^{(n,k)}\le n^p$ by the definition of $T_{n,k}$. From a Gronwall type inequality \cite[Lemma IX.6.3]{jacod2013limit} and the definition of $T_k$ in \eqref{eq:stopping tildeX}, we derive that $\bE[u^{(n,k)}(\infty)]\le C e^{C(k^2+k)}$. Letting $n\to \infty$ yields that 
\[
\big\Vert \cX^{T_k}\big\Vert_{\cS^p}^p = \bE\left[\sup_{0\le t\le T_k} |\cX_t|^p\right] \le C e^{C(k^2+k)}.
\]
Recall that $\cX$ satisfies \eqref{eq:change bbflow}. By Emery's inequality (cf. \cite[Theorem V.3]{protter2005stochastic}) and the linear growth property of $\sigma$, we deduce that 
\[
\big\Vert \cX^{T_k}\big\Vert_{\cH^p} \le  \big\Vert \sigma(\cX^{T_k})\big\Vert_{\cS^p} \big\Vert \tX^{T_k}\big\Vert_{\cH^\infty}\le k(2+2\delta) \big(1+\big\Vert \cX^{T_k}\big\Vert_{\cS^p}\big),
\]
since $\Vert \tX^{T_k}\Vert_{\cH^\infty}\le k$.
The $\cH^p$-boundedness of $\cX^{T_k}$ follows. 

To derive the uniform upper bound for $\cX(x)$ with $|x|\le T$, we observe that $\sigma(\cdot+x)$ is of linear growth with the multiplying constant uniform in $|x|\le T$. Repeating the previous argument on $\cX(x)-x$, we deduce that for all $|x|\le T$, the $\cH^p$ norm of $(\cX(x))^{T_k}$ is uniformly bounded by a constant depending on $\delta,\mu,k,p$ and $T$. This completes the proof.
\end{proof}

Recall the equality $Z_t(x)=\int_0^t \sigma'\big(\cX_s(x)\big)\ind_{\{\cX_s(x)\ne 0\}}\dd \tX_s$ at the beginning of the proof of Proposition~\ref{p:biconti L}.

\begin{lemma}\label{l:biconti Z}
The process $(Z_t(x):t\ge 0,\, x\in s(\bR))$ admits a version such that $(t,x)\mapsto Z_t(x)$ is continuous. 
\end{lemma}

\begin{proof}
Decompose $Z_t(x)=N_t(x)+A_t(x)$ where $N_t(x)$ is a continuous local martingale and $A_t(x)$ a continuous adapted process of finite variation, both starting at 0. 
Recall that the finite variation part of $\tX_t$ is $(\mu-1)L_U(t,0) = (1-\frac{1}{\mu}) M_t(X)$. We deduce from \eqref{eq:dI} that for each $x\in s(\bR)$, almost surely,
\[
A_t(x) = \bigg(\frac1\mu -1\bigg) \big( 2\delta\, M_t(X) + 2\lambda [(M_t(X) -[s^{-1}(x)]^-)\vee 0]\big), \qquad t\ge 0,
\]
because $\sigma'(x)=-2\delta$ for $x<0$ and $\sigma'(x)=-2\delta-2\lambda$ for $x>0$, and $\cX_t(x)>0$ (resp. $\cX_t(x)<0$)  if and only if $\cR_t(s^{-1}(x))>0$ (resp. $\cR_t(s^{-1}(x))<0$).
It follows that $A_t(x)$ admits a bicontinuous version.

It remains to show that there exists a bicontinuous version of $N_t(x)$. Recall the definition of $(T_k)$ in \eqref{eq:stopping tildeX}. The idea is similar to \cite[Lemma 5.1]{hu-warren}: we only need to prove that for every $k\ge 1$, $p>4$ and $T>0$, there exists some constant $C$ such that
\begin{align}\label{eq:M estimate}
    &\bE\bigg[\sup_{0\le t\le T\wedge T_k}|N_t(x)-N_t(y)|^p\bigg]\le C|x-y|^{p/4}
\end{align}
for every $x,y\in s(\bR)$ with $|x|,|y|\le T$. The bicontinuity of $N$ then follows by Kolmogorov's continuity theorem (see e.g. \cite[Theorem 1.4.1]{kunita1990stochastic}).

Suppose $x<y$. From the monotonicity $\cX_t(x) \le \cX_t(y)$ and the explicit form of $\sigma'$, we have
\begin{align*}
    N_t(x) - N_t(y) = 2\lambda \int_0^t \ind_{\{\cX_s(x)<0<\cX_s(y)\}} \dd \beta_s.
\end{align*}
Fix $\omega_0$ such that $\inf_{|x| \le \omega_0} |\sigma(x)| > 0$. Then for $0 < \omega < \omega_0$, by BDG's inequality,
\begin{align*}
    &\bE \left[\sup_{0 \le t \le T\wedge T_k} |N_t(x) - N_t(y)|^p \right]
    \le C_1(p)\, \bE \left[\left( \int_0^{T\wedge T_k} \ind_{\{\cX_s(x)<0<\cX_s(y)\}} \dd s  \right)^{\frac p2} \right]\\
    \le\, & C_1(p)\left[ \bE \left[\left( \int_0^{T_k}  \ind_{\{-\omega \le \cX_s(x)<0\}} \dd s \right)^{\frac p2}\right] 
    + \bE \left[\left( \int_0^T  \ind_{\{|\cX_s(y) - \cX_s(x)| > \omega\}} \dd s \right)^{\frac p2} \right]\right]\\
    \le \, & C_1(p)\left[ \sup_{ |z|\le \omega_0}\frac{1}{|\sigma(z)|^p} \bE \left[ \left( \int_0^{T_k}  \ind_{\{-\omega \le \cX_s(x)<0\}} \dd \langle \cX(x)\rangle_s \right)^{\frac p2}\right] 
    + \bE \left[\left( \int_0^T  \ind_{\{|\cX_s(y) - \cX_s(x)| > \omega\}} \dd s \right)^{\frac p2} \right]\right]\\
    \le\, & C_2(p) \omega^{\frac p2} \big\Vert (\cX(x))^{T_k} \big\Vert_{\cH^{\frac p2}}^{\frac p2} + C_1(p)\, \omega^{-\frac p2} \bE \left[\left( \int_0^T |\cX_s(x) - \cX_s(y)| \dd s \right)^{\frac p2}\right],
\end{align*}
where the first term in the penultimate inequality follows since $\langle \cX(x)\rangle_t = \int_0^t \sigma(\cX_s(x))^2\dd s$ from \eqref{eq:change bbflow}, and the first part in the last inequality follows from \cite[Theorem 4.1]{barlow1981semi} with $(k,q,p)$ replaced by $(\frac p2, 2,2)$ there. By Lemma~\ref{l:bounded H norm}, the first term is bounded by $C_3(p,k,T) \omega^{p/2}$.
On the other hand, by \cite[Theorem V.43]{protter2005stochastic}, for every $r\in\bR$, we have for all $t\ge 0$,
\[
|\cX_t(x)-\cX_t(y)|^r = |x-y|^r\exp\Big(\Lambda_t(x,y)-\frac{1}{2}\langle \Lambda_\cdot(x,y)\rangle_t\Big),
\]
where
\begin{align*}
   \Lambda_t(x,y) &= \int_0^t H_s(x,y)\dd \tX_s + \int_0^t J_s(x,y)\dd s,\\
   H_t(x,y) &= r\,\frac{\sigma\big(\cX_t(x)\big)-\sigma\big(\cX_t(y)\big)}{\cX_t(x)-\cX_t(y)},\\
   J_t(x,y) &= \frac{r(r-1)}{2}\frac{\Big(\sigma\big(\cX_t(x)\big)-\sigma\big(\cX_t(y)\big)\Big)^2}{\big(\cX_t(x)-\cX_t(y)\big)^2}.
\end{align*}
Since $\sigma$ is globally Lipschitz continuous, the two processes $H$ and $J$ are bounded. From \eqref{eq:U-B}, $L_U(T,0)\eqdis M_T(B^\delta)$ which has exponential moments of every order. This fact together with Novikov's criterion and Doob's inequality implies that 
\[
\bE\bigg[\sup_{0\le t\le T}|\cX_t(x) - \cX_t(y)|^r\bigg] \le C_4(r,T)|x - y|^r.
\]
Hence 
\begin{align*}
     \bE \left[\left( \int_0^T |\cX_s(x) - \cX_s(y)|\dd s \right)^{\frac p2}\right] \le T^{\frac p2}\bE\bigg[\sup_{0\le t\le T}|\cX_t(x) - \cX_t(y)|^{\frac p2}\bigg] \le C_5(p,T) |x - y|^{\frac p2}.
\end{align*}
We obtain the desired estimate by taking $\omega = \min(\omega_0, |x - y|^{1/2})$.
\end{proof}

\begin{proposition}\label{p:derivative}
With probability 1, for all $t\ge 0$ and $x\in s(\bR)$,
\begin{align}\label{eq:SDE derivative general}
\frac{\partial \cX_t(x)}{\partial x} = U_t(x):=\exp\Big(Z_t(x) - \frac12 \langle Z_\cdot(x)\rangle_t\Big).
\end{align}
Moreover, $U_t(x)$ admits a version so that $(t,x)\mapsto U_t(x)$ is continuous.
\end{proposition}

\begin{proof}
Recall the decomposition $Z_t(x)=N_t(x)+A_t(x)$. We observe that $|\sqrt{\langle N_\cdot(x)\rangle} - \sqrt{\langle N_\cdot(y)\rangle}|\le \sqrt{\langle N_\cdot(x) - N_\cdot(y)\rangle}$. Therefore by BDG's inequality, $|\sqrt{\langle N_\cdot(x)\rangle_t} - \sqrt{\langle N_\cdot(y)\rangle_t}|$ admits the same $L^p$-estimate as the right-hand side of \eqref{eq:M estimate} with a larger multiplying constant. It follows that $\langle Z_\cdot(x)\rangle_t = \langle N_\cdot(x)\rangle_t$ and hence $U_t(x)$ admit a bicontinuous version. Therefore, it suffices to prove that for all $t\ge 0$ and $h\in C_c^\infty(s(\bR))$,
\begin{align}\label{eq:def generalized derivative}
\int_\bR h(x)U_t(x)\dd x = - \int_\bR h'(x)\cX_t(x)\dd x.
\end{align}
We approximate $\sigma$ by $\sigma_n\in C^{\infty}(\bR)$ such that $\sigma_n(0)=\sigma(0)$, $(\sigma'_n: n\ge 1)$ is uniformly bounded on $\bR$ and $\sigma'_n\equiv \sigma'$ on $\bR\, \setminus(-1/n,1/n)$. This choice ensures that $\lim_{n\to \infty}\sup_{x\in\bR} |\sigma_n(x)-\sigma(x)|=0$.
Consider the equation \eqref{eq:change bbflow} with $\sigma$ replaced by $\sigma_n$. By \cite[Theorem V.49]{protter2005stochastic}, it induces a stochastic flow $(\cX^{(n)}_t(x):t\ge 0,\, x\in\bR)$ such that for every $t\ge 0$, $x\mapsto \cX^{(n)}_t(x)$ is differentiable, with its derivative satisfying the equation
\begin{align}\label{eq:SDE derivative}
    U_t^{(n)}(x) = 1 +\int_0^t \sigma_n'\big(\cX_s^{(n)}(x)\big) U_s^{(n)}(x) \dd \tX_s.
\end{align}
Therefore \eqref{eq:def generalized derivative} holds with $(U_t(x),\cX_t(x))$ replaced by $(U^{(n)}_t(x),\cX^{(n)}_t(x))$. Set 
\begin{align}\label{eq:Z^n(x)}
Z^{(n)}_t(x) := \int_0^t\sigma'_n\big(\cX_s^{(n)}(x)\big)\dd \tX_s.
\end{align}
Then \eqref{eq:SDE derivative} implies that $U^{(n)}_t(x) = \exp(Z^{(n)}_t(x)-\frac12 \langle Z^{(n)}(x)\rangle_t)$.

We first show that $\int_\bR |h'(x)| |\cX_t^{(n)}(x) - \cX_t(x)| \dd x$ converges in probability to 0. 
We use the technique of localization together with Emery's inequality.
To be specific, let $(T_k:k\ge 1)$ be defined via \eqref{eq:stopping tildeX}. 
Following (an even simpler version of) the proof of \cite[Theorems V.10--V.11]{protter2005stochastic}, we deduce that for all $k\ge 1$,
\begin{equation}\label{eq:converge locally}
\begin{aligned}
\sup_{x\in\bR} \Big\Vert \big(\cX^{(n)}(x)-\cX(x)\big)^{T_k} \Big\Vert_{\cS^2}
&\le C(k)\, \sup_{x\in\bR} \bigg\Vert \int_0^{\cdot \wedge T_k} [\sigma_n(\cX_s(x))-\sigma(\cX_s(x))] \dd \tX_s \bigg\Vert_{\cS^2} \\
& \le C(k)\,\Vert \tX^{T_k}\Vert_{\cH^2}\, \sup_{x\in\bR}|\sigma_n(x)-\sigma(x)|\to 0,\qquad n\to \infty.
\end{aligned}
\end{equation}
We observe that for every $T,\varepsilon>0$ and $k\in\bN_+$,
\begin{align}\label{eq:convergence prob}
\sup_{x\in\bR}\ \bP \bigg(\sup_{0\le t\le T} \Big|\cX^{(n)}_t(x)-\cX_t(x)\Big|>\varepsilon \bigg) \le \frac1{\varepsilon^2} \sup_{x\in\bR} \Big\Vert \big(\cX^{(n)}(x)-\cX(x)\big)^{T_k} \Big\Vert_{\cS^2}^2 + \bP(T_k\le T).
\end{align}
Sending first $n\to \infty$ and then $k\to \infty$ yields that the left-hand side of \eqref{eq:convergence prob} converges to 0 as $n\to \infty$, which implies the desired convergence shown at the beginning. 

It remains to prove that $\int_\bR |h(x)| |U_t^{(n)}(x) - U_t(x)| \dd x$ converges in probability to 0.
We first show that $Z_t^{(n)}(x)$ converges to $Z_t(x)$ in probability for all fixed $t\ge 0$ and $x\in s(\bR)$. Recall the definition of $Z_t(x)$ and $Z_t^{(n)}(x)$ in Lemma~\ref{l:biconti Z} and \eqref{eq:Z^n(x)} and the relation $\tX_t = \beta_t +(\mu-1) L_U(t,0)$.
For every $x\in s(\bR)$, consider two events $A_{n,\varepsilon}:=\{\sup_{0\le s\le t} \big|\cX^{(n)}_s(x)-\cX_s(x)\big|\le \frac\varepsilon2\}$ and $B_{\varepsilon,s}:=\{\big|\cX_s(x)\big|>\varepsilon\}$. Then we have $\big|\cX_s^{(n)}(x)\big|>\frac\varepsilon2$ on the event $A_{n,\varepsilon}\cap B_{\varepsilon,s}$. Since $\sigma_n'(x) = \sigma'(x)$ for $|x| \ge 1/n$, we deduce that for all $n$ large enough, $\sigma'_n(\cX_s^{(n)}(x)) = \sigma'(\cX_s(x))$ on the event $A_{n,\varepsilon}\cap B_{\varepsilon,s}$. Therefore, bounding $\sigma'$ and $\sigma'_n$ by some constant $M$ on the event $A_{n,\varepsilon}^c\cup B_{\varepsilon,s}^c$, we derive that for every $t\ge 0$,
\[
\bE\bigg[\int_0^t \big[\sigma_n'(\cX_s^{(n)}(x)) - \sigma'(\cX_s(x)) \ind_{\{\cX_s(x) \ne 0\}}]^2  \dd s\bigg]\le M\bigg(t\,\bP(A_{n,\varepsilon}^c)+\int_0^t \bP(B_{\varepsilon,s}^c)\dd s\bigg).
\]
Since $\langle \cX(x)\rangle_t = \int_0^t \sigma(\cX_s(x))^2\dd s$ and $\sigma(0)=1$, we have $\int_0^t \ind_{\{\cX_s(x)=0\}}\dd s =0$ by the occupation times formula. Letting $n\to \infty$ and applying \eqref{eq:convergence prob}, and then letting $\varepsilon\to 0$ and applying the dominated convergence theorem, we deduce that
\[
\bE\bigg[\int_0^t \big(\sigma_n'(\cX_s^{(n)}(x)) - \sigma'(\cX_s(x)) \ind_{\{\cX_s(x) \neq 0\}}\big)^2  \dd s\bigg] \to 0,\qquad n\to \infty.
\]
We can now apply \eqref{eq:dI 0} and analogously deduce that, for all $k\ge 1$ and $t\ge 0$,
\begin{align*}
    &\bE\bigg[\int_0^{t\wedge T_k} \big(\sigma_n'(\cX_s^{(n)}(x)) - \sigma'(\cX_s(x))\ind_{\{\cX_s(x) \neq 0\}}\big)^2 \dd L_U(s,0)\bigg]\to 0, \qquad n\to \infty.
\end{align*}
Recall that $T_k\le k^2$. It follows that for all $k\ge 1$, $(Z^{(n)}(x))^{T_k}$ converges to $(Z(x))^{T_k}$ in $\cH^2$. Hence $Z_t^{(n)}(x)$ converges to $Z_t(x)$ in probability by \cite[Theorem V.12]{protter2005stochastic}, and thus $U_t^{(n)}(x)$ converges in probability to $U_t(x)$. Recall that $L_U(t,0)\eqdis M_t(B^\delta)$ has exponential moments of every order. Together with Novikov's criterion and Doob's inequality, it yields that for every $t>0$ and compact set $K\subseteq s(\bR)$,
\[
\sup_{x \in K, \, n \ge 1} \bE \Bigg[\bigg( \sup_{0 \le s \le t} \big[|U_s^{(n)}(x)| + |U_s(x)|\big] \bigg)^2 \Bigg]< \infty.
\]
Hence $\{U^{(n)}_s(x),\, U_s(x): 0\le s\le t,\, x\in K\}$ is uniformly integrable.
We deduce the desired convergence, hence also \eqref{eq:def generalized derivative}.
\end{proof}

\bibliographystyle{abbrvurl}
\bibliography{sample}

\end{document}